\documentclass[a4paper,10pt]{amsart}

\usepackage{amsthm}                   %Theorem environment
\usepackage{amsmath}                  %Math support
\usepackage{amsfonts}                 %Math fonts
\usepackage{amssymb}                  %Additional symbols
\usepackage{mathrsfs}                 %Better script letters
\usepackage{hyperref}                   % To index references
\usepackage{xy}
\usepackage{graphicx,subfigure}
\xyoption{all}

\usepackage{epstopdf}

\theoremstyle{definition}
\newtheorem{definition}{Definition}[section]
\newtheorem{remark}[definition]{Remark}
\newtheorem{example}{Example}[section]

\theoremstyle{plain}
\newtheorem{theorem}[definition]{Theorem}
\newtheorem{lemma}[definition]{Lemma}
\newtheorem{corollary}[definition]{Corollary}
\newtheorem{proposition}[definition]{Proposition}
\newtheorem*{theorem*}{Theorem}

\numberwithin{equation}{section}

\newcommand{\Z}{\ensuremath{\mathbb{Z}}}     %Integers
\newcommand{\bomega}{{\boldsymbol {\omega}}}
\newcommand{\bgamma}{{\boldsymbol {\gamma}}}

\newcommand{\mi}{{\mathbf{i}}}
\newcommand{\cS}{{\mathcal{S}}}
\newcommand{\ignore}[1]{}

\def\cal{\mathcal }

\def\Z{\mathbb Z}

\begin{document}
\author{Hui Rao}
\address{Hui Rao: Department of Mathematics and Statistics, Central China Normal University, CHINA}
\email{hrao@mail.ccnu.edu.cn}
\author{Shu-Qin Zhang $\dag$}
\address{ Shu-Qin Zhang: Chair of Mathematics and Statistics, University of Leoben, Franz-Josef-Strasse 18, A-8700 Leoben, Austria}

\email{zhangsq\_ccnu@sina.com}
\title[Skeleton of self-similar sets]{Space-filling curves of self-similar sets (III): Skeletons}
%\dedicatory{Dedicated to Professor Robert Tichy on the occasion of his 60$^{\,th}$ birthday}
%\email{joerg.thuswaldner@unileoben.ac.at}

\thanks{$\dag$ The corresponding author.}
\thanks{The first author is supported by NSFC-11971195 and NSFC-11431007 and the second author is 
supported by project I1136 and by the doctoral program W1230 granted by the Austrian Science Fund (FWF)}
\date{\today}

\maketitle

\allowdisplaybreaks

\begin{abstract}
Skeleton is a new notion designed for constructing space-filling curves of self-similar sets.
In a previous paper by Dai and the authors \cite{DaiRaoZhang15}, it was shown that  for all connected self-similar sets with a skeleton satisfying the open set condition, space-filling curves can be constructed.  In this paper,  we give a criterion of existence of skeletons by using the so-called neighbor graph of a self-similar set.
 In particular, we show that a connected self-similar set satisfying the finite type condition always possesses skeletons: an algorithm is obtained here.
\end{abstract}

\section{Introduction }
Space-filling curves (SFC) have attracted the attention of mathematicians  over a century since Peano's seminal work \cite{Peano1890}.
 In a series of three papers,  \cite{RaoZhangS16}, \cite{DaiRaoZhang15} and the present paper, we  give a systematic investigation of  space-filling curves  for connected self-similar sets.

The notion of skeleton, which can be regarded as a kind of vertex set of a fractal,  was first introduced in  \cite{DaiWang10}, designed for SFCs of self-affine tiles.
%which is  close to the ideas used in [BS1990] and in [Str 1999] for building minimal paths in PCF self-similar sets.
 The constructions of SFCs in \cite{RaoZhangS16} and \cite{DaiRaoZhang15}
  are based on the assumption that the self-similar set in consideration possesses a skeleton. Precisely, it is shown  that

   \begin{theorem}[\cite{DaiRaoZhang15}] \label{old} 
  Let $K$ be a connected self-similar set which has a skeleton and satisfies the open set condition, then $K$ admits space-filling curves.
  \end{theorem}

\begin{figure}[htpb]
\subfigure[]{\includegraphics[width=0.3 \textwidth]{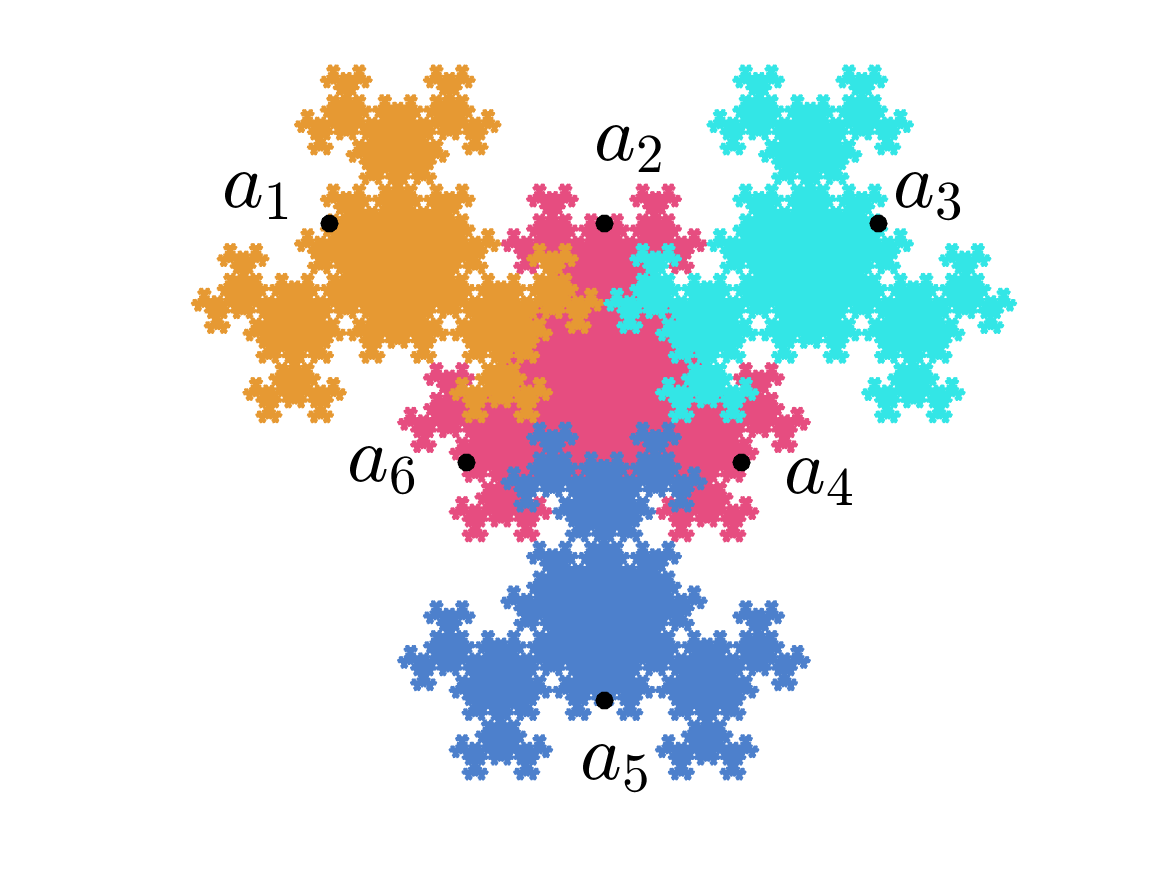}}\subfigure[]{\includegraphics[width=0.3 \textwidth]{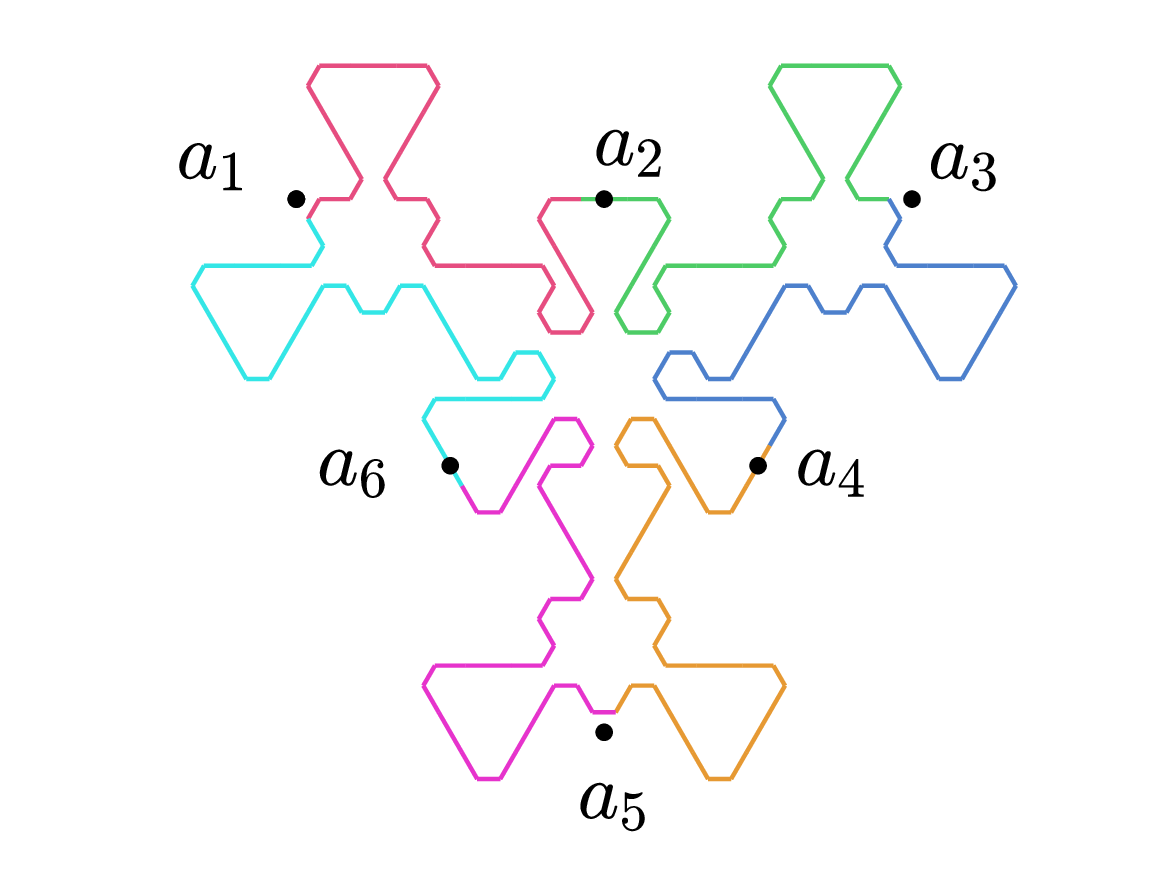}}\subfigure[]{\includegraphics[width=0.3 \textwidth]{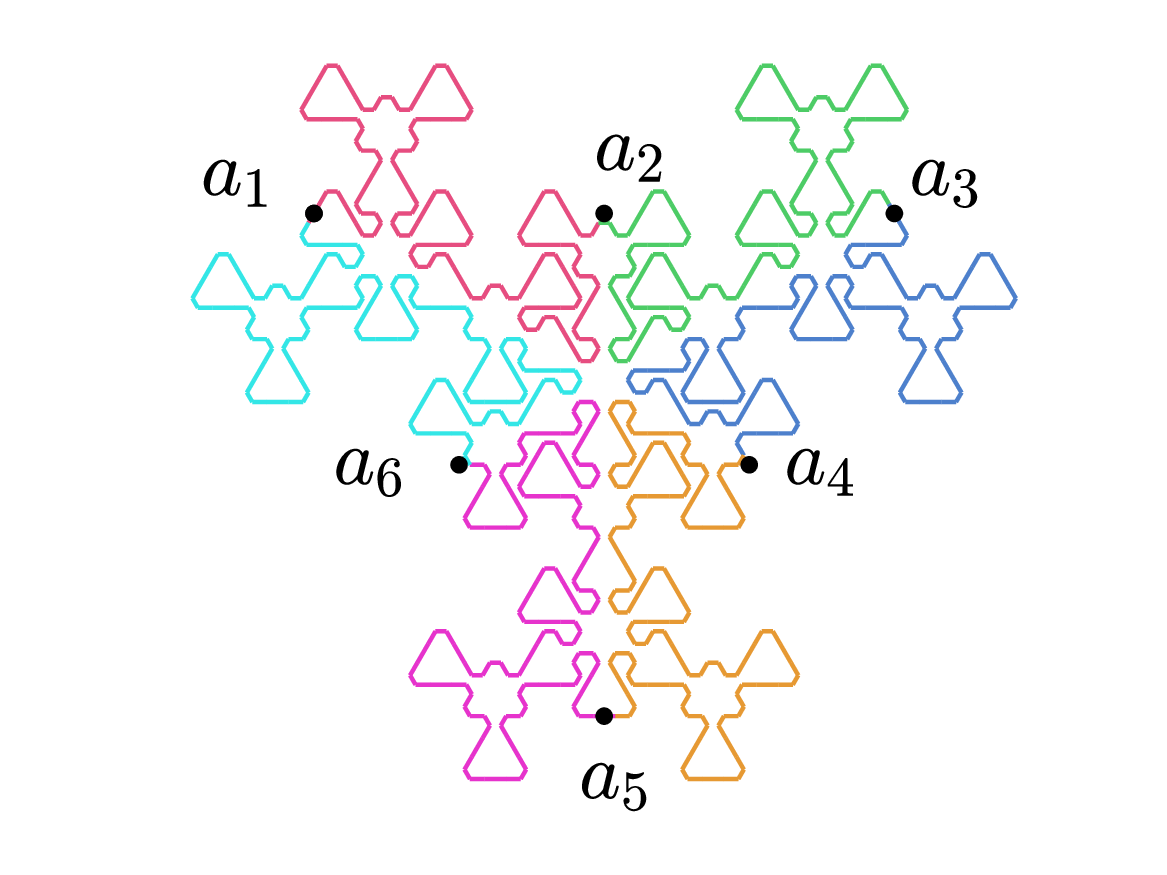}}
\caption{(a) illustrates the four-tile star and $\{a_1,\dots, a_6\}$, a skeleton of it. (b), (c) are the approximating curves constructed by the positive Euler-tour method in \cite{DaiRaoZhang15} with
this skeleton. For more details, see \cite[Section 7]{RaoZhangS16}.}\label{SFC-four-tile}
\end{figure}

\begin{figure}[htpb]
\subfigure[]{\includegraphics[width=0.3 \textwidth]{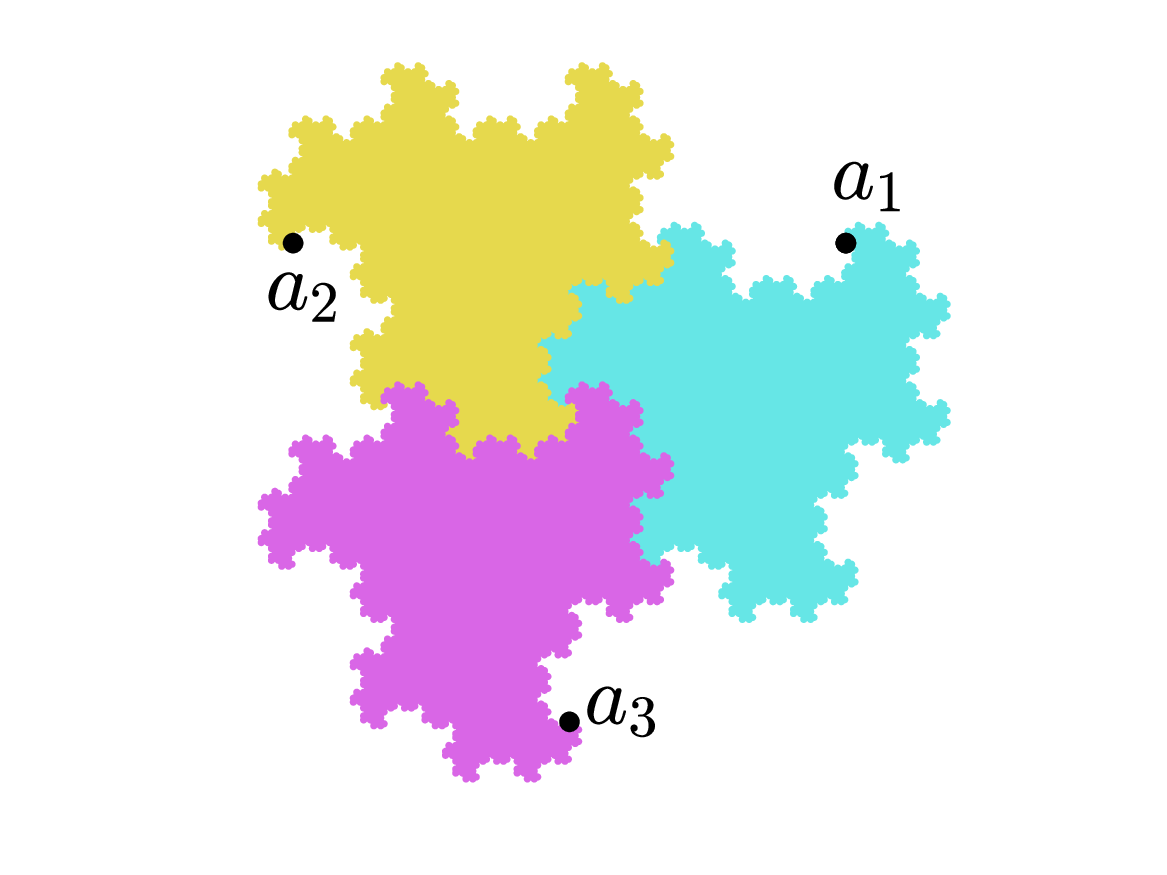}}\subfigure[]{\includegraphics[width=0.3\textwidth]{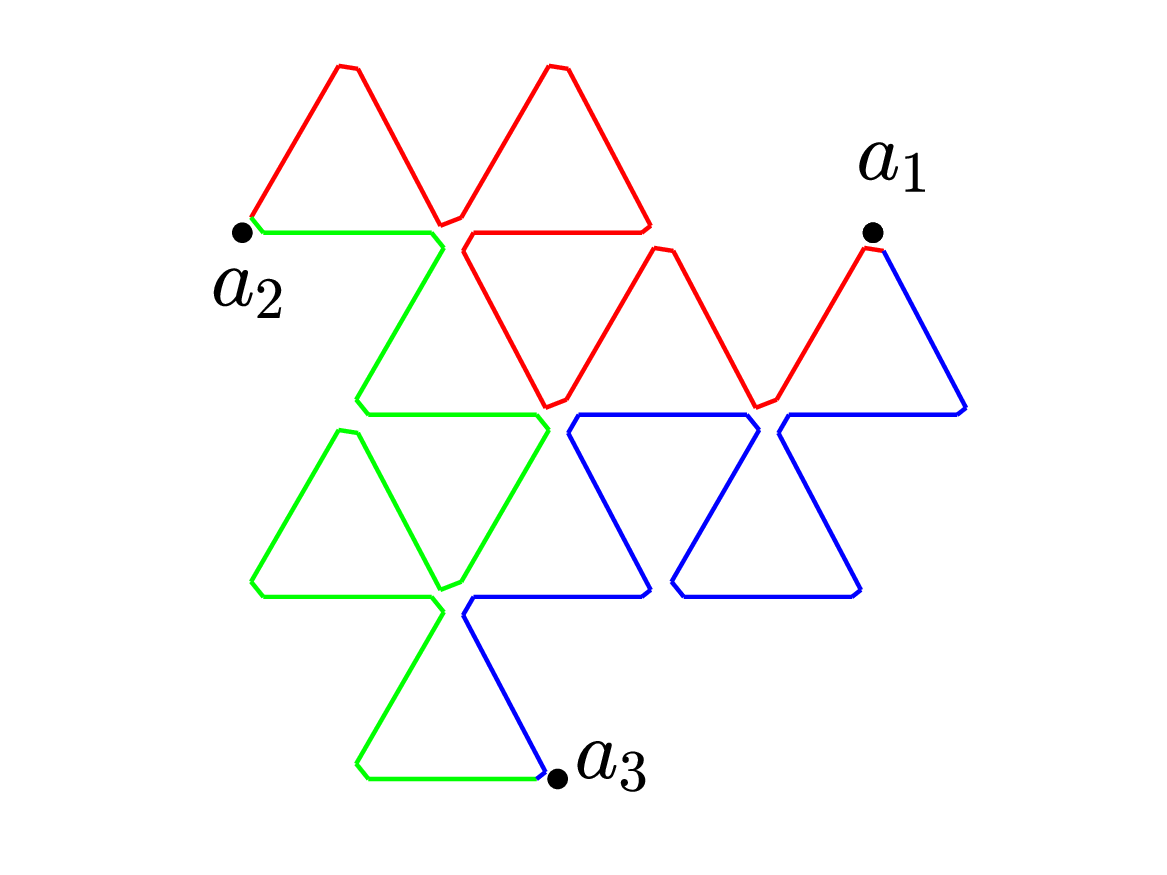}}\subfigure[]{\includegraphics[width=0.3 \textwidth]{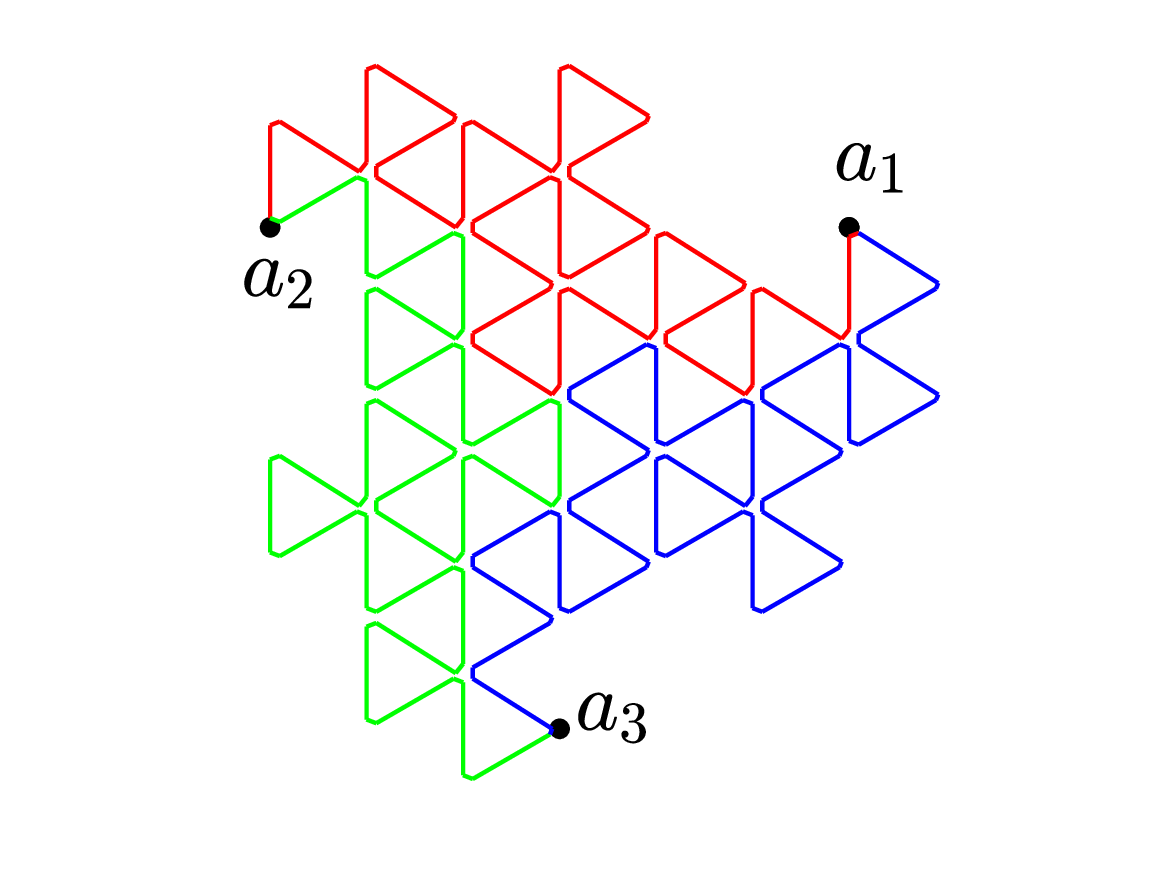}}
\caption{(a) shows the fractal Terdragon and  $\{a_1, a_2, a_3\}$, a skeleton of it. (b), (c) are the approximating curves constructed in \cite[Example 5.1]{DaiRaoZhang15} with this skeleton.}\label{SFC-terdragon}
\end{figure}

The goal of the present paper is to study when a self-similar set has skeletons and how to find them, which  is the last part of our theory on constructing SFCs.

Recall that a \emph{self-similar set} is a  non-empty compact set $K$ satisfying the set equation
$$K=\bigcup_{i=1}^N S_i(K),$$
where $S_1,\dots,S_N$ are contraction similitudes on $\mathbb{R}^d$.
The family $\{S_1,\dots, S_N\}$ is called an \emph{iterated function system}, or IFS in short; $K$ is also called the invariant set of the IFS. (See for instance, \cite{Hutchinson81, Falconer90}). For  the open set condition, we refer to \cite[Section 9.2]{Falconer90}.

To define the skeleton of a self-similar set, we  construct a  graph which is a generalization of Hata \cite{Hata85}.
Let $\{S_j\}_{j=1}^{N}$ be an IFS with invariant set   $K$.
For any subset $A$ of $K$, we define an undirected graph $H(A)$ as follows:
\begin{itemize}
\item[(i)] The vertex set is $V=\{S_1,S_2,\dots,S_N\}$;
\item[(ii)]  There  is an edge between   $S_i$ and $S_j$ $(i\neq j)$ if and only if $S_i(A)\cap S_j (A)\neq\emptyset$.
\end{itemize}
We call $H(A)$  the \emph{Hata graph} induced by $A$ (with respect to the IFS $\{S_1,\dots, S_N\}$).

 \begin{figure}[htbp]
\subfigure[\text{Sierpi\'nski carpet} $C$]{\includegraphics[width=4 cm]{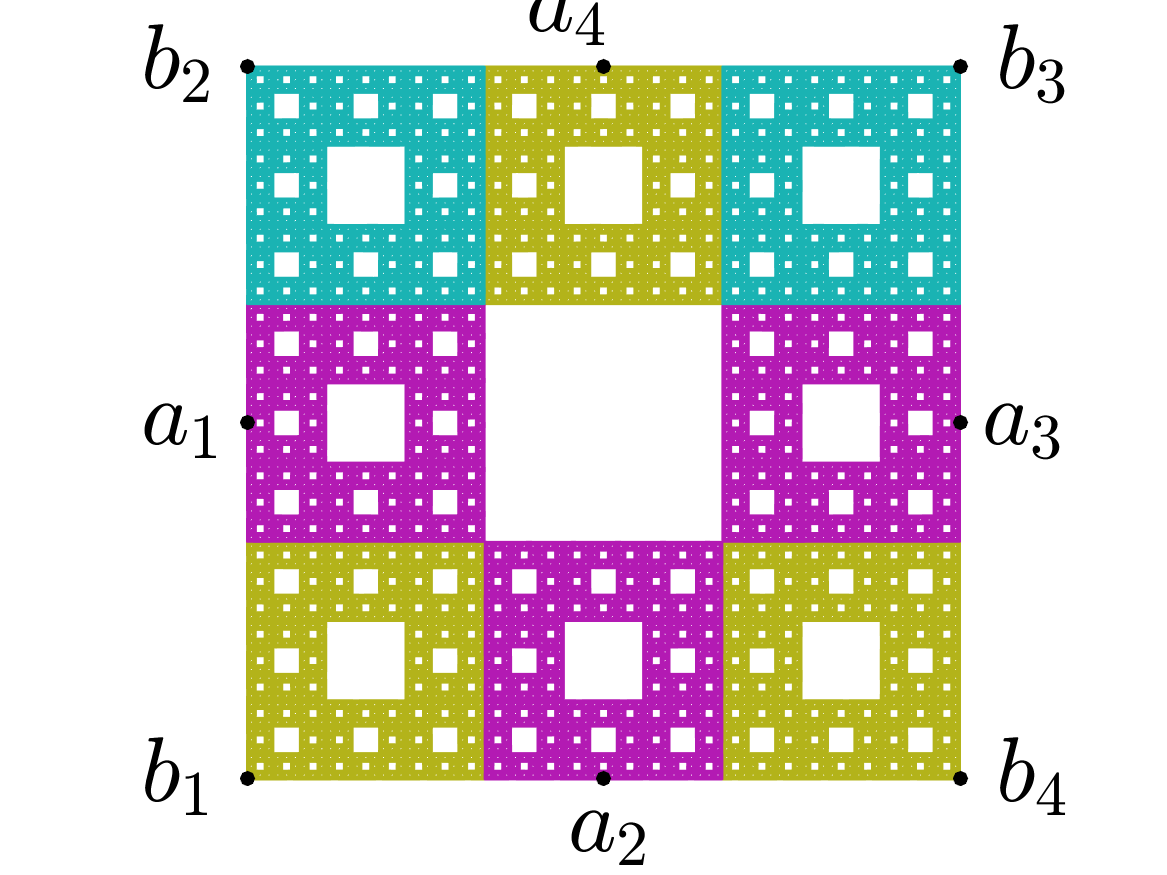}}\quad\quad\quad
\subfigure[$H(A)$ with $A=C$ or $A=\{b_1, b_2, b_3, b_4\}$]{\includegraphics[width=3.1  cm]{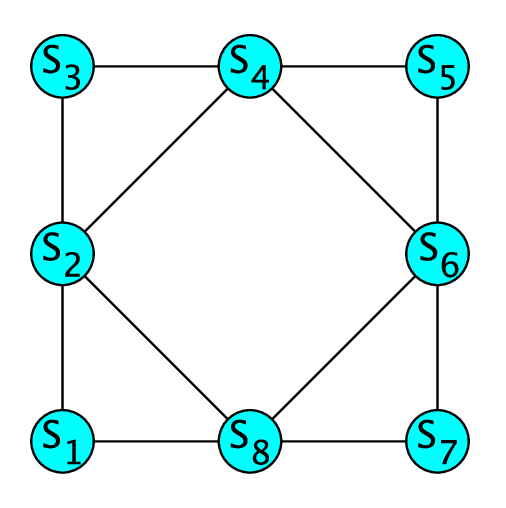}}\quad\quad\quad
\subfigure[$H(A)$ with $A=\{a_1, a_2, a_3, a_4\}$]{\includegraphics[width=3.1 cm]{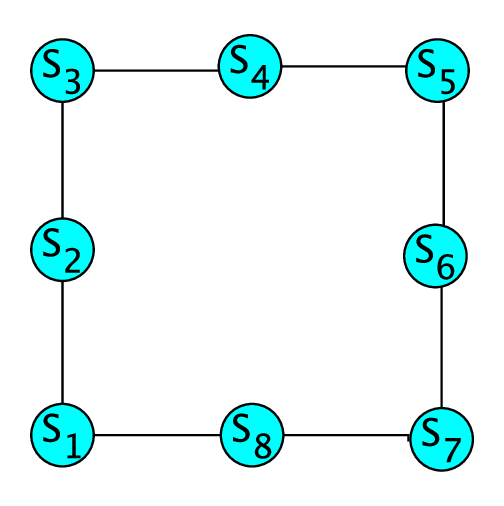}}
%\subfigure[$H(A): A=\{b_1, b_3\}$]{\includegraphics[width=3.1 cm]{Carpet_Hata3.eps}}
\caption{The Hata graphs  $H(A)$ of the Sierpi\' nski carpet w.r.t. the IFS $\{S_i(x)=\frac{x+d_i}{3}\}_{i=1}^8$, $\{d_1,\dots, d_8\}=\{0,\mi,2\mi,2\mi+1, 2\mi+2, \mi+2, 2, 1 \}$. It is easy to check that both $\{a_1,\dots, a_4\}$ and $\{b_1,\dots, b_4\}$   are skeletons of  Sierpi\'nski carpet.}
\label{Hata_gra1}
\end{figure}

\begin{remark}   Hata \cite{Hata85} introduced
the above graph but fixed $A$ to be $K$.
Hata proved that a self-similar set $K$ is connected if and only if the graph $H(K)$ is connected.
 (An undirected graph $G$ is said to be   \emph{connected}, if for every pair of vertices in $G$  there is a path in $G$ joining them.)
\end{remark}

\begin{definition}\label{def-skeleton}
Let $\{S_j\}_{j=1}^N$ be an IFS such that the invariant set $K$ is connected.  We call a finite subset $A$ of $K$ a \emph{skeleton} of $\{S_j\}_{j=1}^{N}$  (or  $K$), if the following two conditions are fulfilled:
\begin{itemize}
\item[(1)] $A$ is stable under iteration, that is, $A \subset \bigcup_{j=1}^N S_j(A)$;
\item[(2)] The Hata graph $H(A)$ is connected.
\end{itemize}
\end{definition}
A skeleton consists of at least two elements if the self-similar set $K$ is not a singleton (Proposition \ref{single}).

In the previous works, there are two notions closely related to
the skeleton, the self-similar zipper and the boundary of a self-similar set, see Remark \ref{Relate-ske_1} and \ref{Relate-ske_2}.

\begin{remark} \label{Relate-ske_1}
%The word {\em zipper} was first used in fractal context by Thurston \cite{Thurston85}. Later Astala introduced  the notion {\em self-similar zippers} in \cite{Astala88}, and then it also appeared in  \cite{Hutchinson81, Tricot94, WenXi03}.  The self-similar zipper defined in Section \ref{zippers} was initially proposed by  Aseev in  \cite{Aseev03}.  Graph-directed version of zippers, called multizippers, were introduced in \cite{Tetenov06}, and it is essentially a linear GIFS studied in  \cite{RaoZhangS16}.

The word zipper was first used in fractal context by Thurston \cite{Thurston85} to denote curves in C which have conformally rigid complement.
Later Astala \cite{Astala88} proved that this property is true for any self-similar quasi-interval in $\mathbb C$ which is not a line segment.
The definition of self-similar zipper with binary signature in Section \ref{zippers} was initially proposed by Aseev and Tetenov \cite{Aseev03}. Before that, self-similar zippers of zero signature (without giving them any name) appeared in several papers, including \cite{Hutchinson81, Tricot94, WenXi03}.
The graph directed version of zippers, called multizippers, was introduced and studied in \cite{Tetenov06}. The definition of multizipper is equivalent to the one of linear GIFS proposed by the authors in \cite{RaoZhangS16,DaiRaoZhang15}.

Self-similar zippers have skeletons of cardinality $2$ (see Section \ref{zippers}).
Indeed, many beautiful fractals have self-similar zipper structures,
for example, the Heighway  dragon and the Gosper island.
It is shown \cite{DaiRaoZhang15} that a self-similar zipper,
which is called the path-on-lattice IFS in \cite{RaoZhangS16}, admits space-filling curves provided the open set condition holds. The website \cite{Jeffrey} provides a nice collection of SFCs of self-similar zippers.
\end{remark}
\begin{figure}[htpb]
\subfigure[]{\includegraphics[width=0.3 \textwidth]{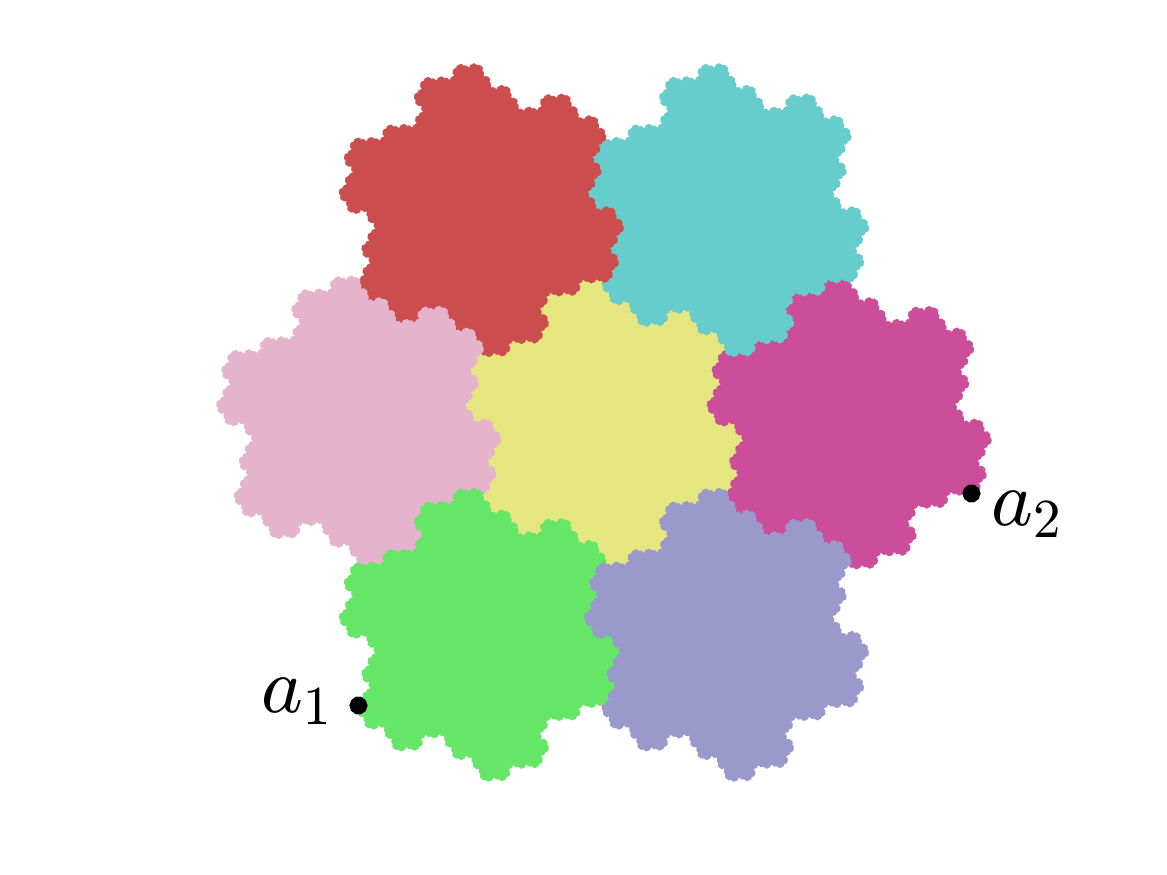}}\subfigure[]{\includegraphics[width=0.3 \textwidth]{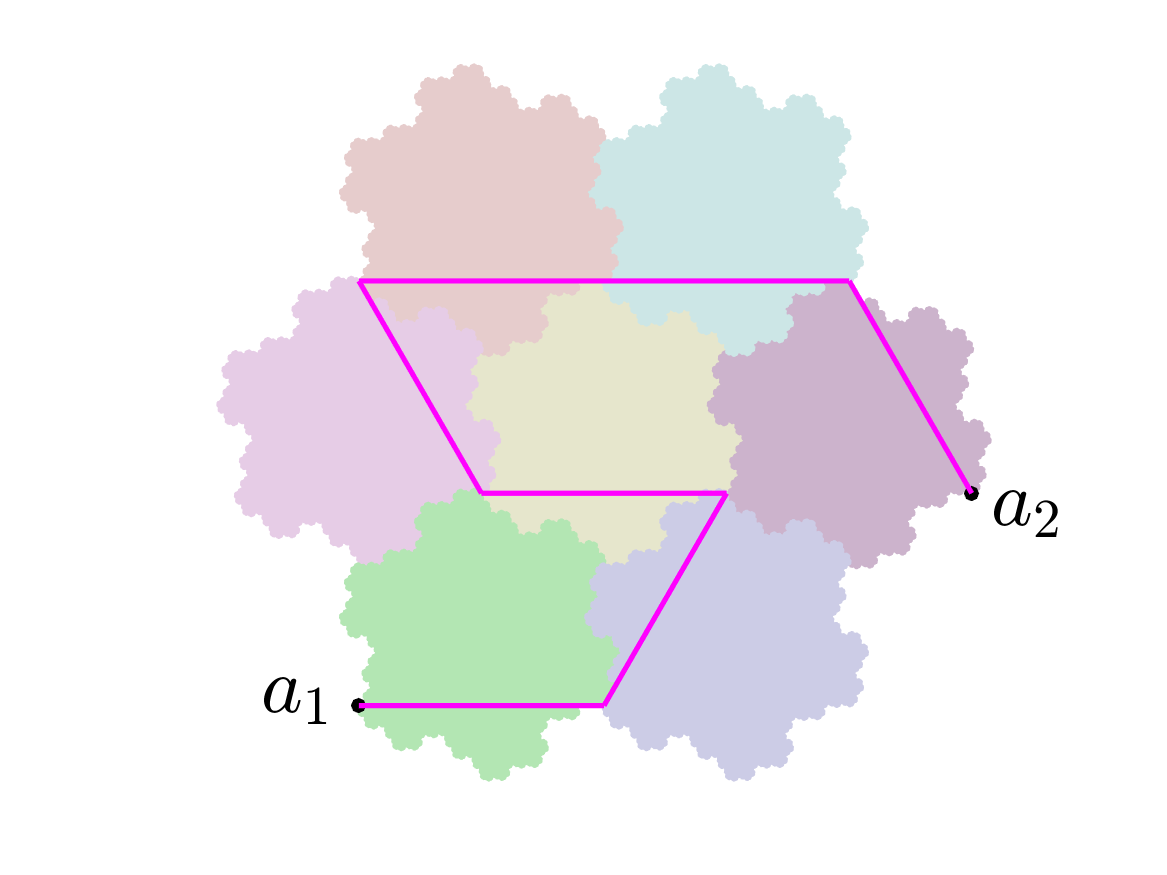}}
\subfigure[]{\includegraphics[width=0.3 \textwidth]{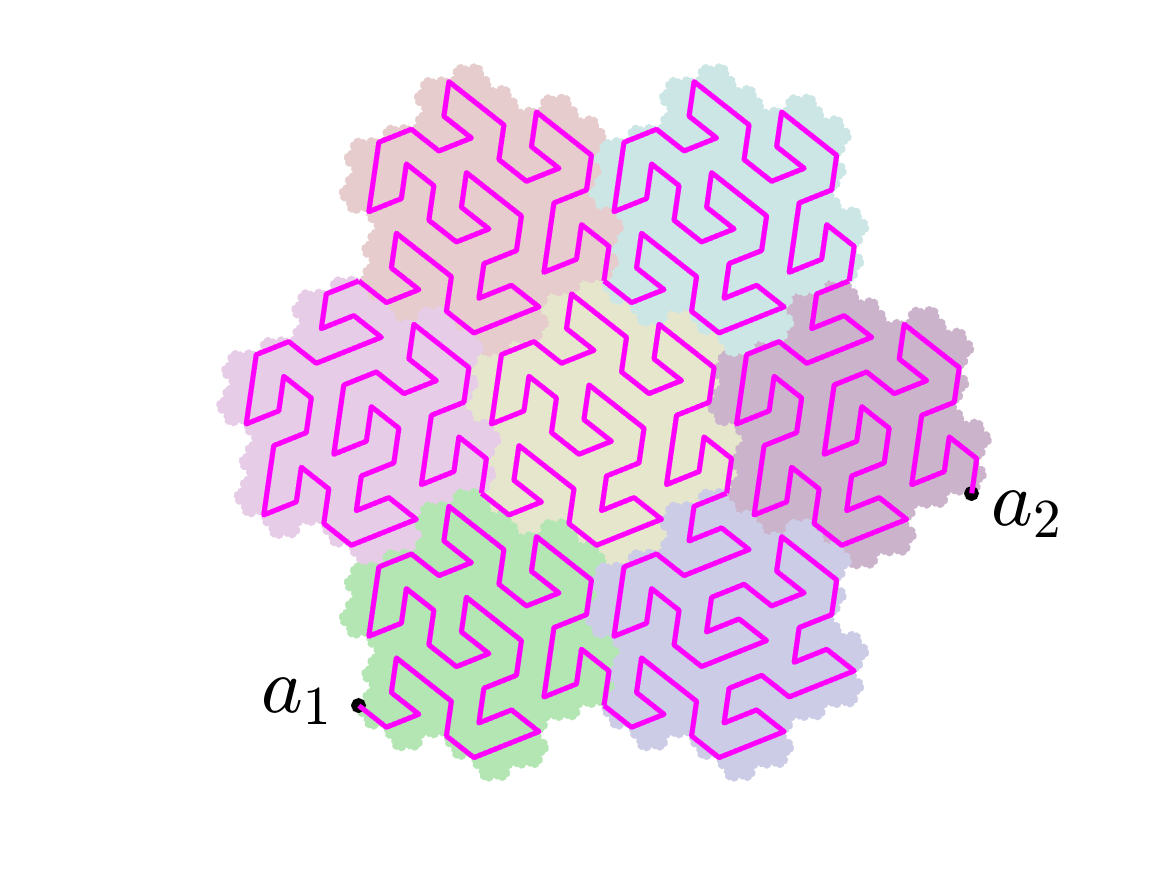}}
\caption{(a) is the Gosper island and a skeleton $\{a_1, a_2\}$. (b), (c)  are the approximating curves generated by path-on-lattice IFS, see \cite[Example 5.2]{RaoZhangS16}.}\label{SFC-Gosper}
\end{figure}
 \begin{remark}\label{Relate-ske_2}
 Kigami \cite{Kigami01} and Mor\'an  \cite{Moran99} have studied the `boundary'  of a self-similar set. If the `boundary' is a finite set,  it is so-called p.c.f. self-similar set (\cite[Definition 1.12]{Kigami93}).  Clearly a p.c.f. self-similar set always has skeletons.  
We also note that the notion of skeleton is very close to the ideas used in \cite{BandtStahnke90, Strichartz99} for building minimal paths in p.c.f. self-similar sets.
%, which can also be seen from the proof of  "if " part of Theorem \ref{Main: skeleton}
 \end{remark}

Next, we show existence of the skeleton for a more general case.
The self-similar sets satisfying \emph{finite type condition} constitute an important class of fractals (see for instance, \cite{RaoWen98, NgaiWang01, BandtMesing09}).
In fact, all the space-filling curves appeared in the literature were constructed for the self-similar sets of finite type. We prove that

\begin{theorem}\label{Main:res}
Let $K$ be the invariant set of an IFS satisfying the finite type condition. If $K$ is connected, then $K$ possesses a skeleton.
\end{theorem}

\begin{proof} This is an immediate consequence of  Theorem \ref{Main: skeleton}  and Theorem \ref{Main: finite-Type}.
\end{proof}

As a consequence of Theorem \ref{old} and Theorem \ref{Main:res}, we conclude that:
\emph{if a connected self-similar set satisfies both the open set condition and the finite type condition, then it admits space-filling curves,} which is the main result of  \cite{DaiRaoZhang15}.

There do  exist   self-similar sets  without skeletons.

\begin{example}\label{Nonske}
\begin{figure}[htbp]
  \centering
  \includegraphics[width=0.35\textwidth]{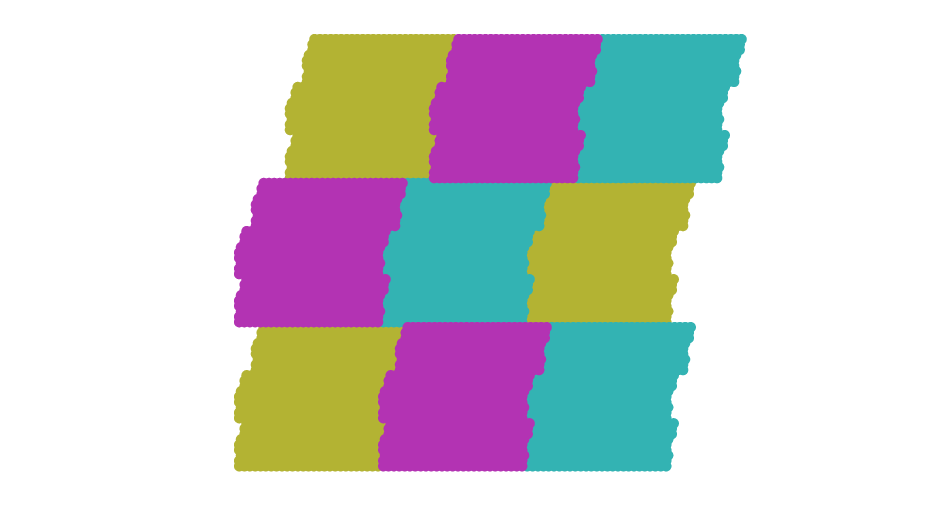}
  \caption{Self-similar set without skeleton}
\end{figure}
{\rm
Let $\mathcal{D}=\{d_1, d_2,\dots, d_9\}=\{ 0, 1, 2, \mi, 1+\mi, 2+\mi, \varepsilon+2\mi, \varepsilon+1+2\mi,
           \varepsilon+2+2\mi\} $
where $\varepsilon$ is a small irrational number, for example, $\varepsilon={\sqrt{2}}/{4}$.
   Kenyon \cite{Kenyon92}  shows that the self-similar set $K$ generated by  the  IFS $\{S_i(z)=(z+d_i)/3\}_{i=1}^9$ is a reptile. In Section \ref{sec:pair} we will show that $K$ has no skeleton.

}\end{example}

This paper is organized as follows. In Section \ref{Sec:basic}, we  give some notation and basic properties of skeletons. In Section \ref{sec:pair}, we introduce the bifurcation pair and give a criterion for the existence
 of skeletons. The finite type condition is discussed in Section \ref{Sec:finitetype}. In the last section we give an algorithm  for finding skeletons  and we provide some examples.

\section{\textbf{Basic properties of skeletons}}\label{Sec:basic}
In this section, we deduce some basic properties of  skeletons.
Let $\cS=\{S_j\}_{j=1}^N$ be an IFS and $K$ be the invariant set.

\subsection{Symbolic space} First, we recall some notations of symbolic space.  Denote $\Sigma=\{1,2,\dots, N\}$ and $\mathbb N=\{0,1,\dots\}$.
For $m\geq1$, define $\Sigma^m=\{1,2\cdots N\}^m$ with $\Sigma^0=\{\emptyset\}$, and set $\Sigma^\ast=\bigcup_{m=0}^\infty{\Sigma^m}$.  Denote  $\Sigma^{\mathbb{N}}$
 the collection of infinite sequences over $\Sigma$.
For $\bomega=(\omega_j)_{j=1}^\infty \in \Sigma^{\mathbb{N}}$, let $\bomega|_n=\omega_1\omega_2\dots\omega_n$ be the prefix of $\bomega$ of length $n$.
 For $i_1i_2\dots i_m\in \Sigma^*$, we  call
$$[i_1i_2\dots i_m]=\{\bomega\in \Sigma^{\mathbb{N}} ~;~~ \bomega|_m=i_1i_2\dots i_m\}$$
a \emph{cylinder}. We denote  $S_{i_1i_2 \dots i_m}=S_{i_1}\circ S_{i_2} \circ  \cdots \circ S_{i_m}$.

Define $\pi:\Sigma^{\mathbb{N}}\rightarrow K$ as
$$\{\pi(\bomega)\}=\bigcap_1^{\infty} S_{\bomega|_m}(K).$$
Then $\pi$ is a surjection. We call $\bomega$ a \emph{coding} of $x$ if $\pi(\bomega)=x$.

The shift map $\sigma:\Sigma^{\mathbb{N}}\rightarrow \Sigma^{\mathbb{N}}$ is defined as
$\sigma((\omega_k)_{k\geq 1})=(\omega_k)_{k\geq 2}$.

 For $k\in \Sigma$, we define
 $\sigma_k:\Sigma^{\mathbb{N}}\rightarrow \Sigma^{\mathbb{N}}$ as
 $$\sigma_k(\omega_1\omega_2\omega_3\dots)=k\omega_1\omega_2\omega_3\dots.$$
 It holds that ~$\pi\circ\sigma_{k}=S_k\circ\pi$ (see \cite{Hutchinson81} or \cite{Kigami01}.)

\subsection{Skeleton of an iteration of an IFS}  Denote the Hata graph of $K$ by $H=H(K)$. We shall denote by $\{S_i,S_j\}$ the edge in $H$ connecting $S_i$ and $S_j$. In an undirected graph $G$, we will use a sequence of distinct vertices in $G$ to indicate a path,
if any two consecutive vertices in the sequence are joined by an edge.

We define the $n$-th \emph{iteration} of $\mathcal{S}$ to be the IFS
 $$\cS^n=\{S_I;~I \in \Sigma^n\}.$$
It is well-known  that  $\cS^n$  shares the same invariant sets as $\cS$ (see Falconer \cite{Falconer90}).
Similarly, we have

\begin{proposition}\label{iter}
 If $A$ is a skeleton of $\mathcal{S}$, then $A$ is also a skeleton of $\mathcal{S}^n$.
\end{proposition}

\begin{proof}
Note that  $H(\cS^n, A)$ is the Hata graph of $A$ related to $\cS^n$.
We prove that
$A\subset \bigcup_{I\in \Sigma^{k}} S_{I}(A)$
 and $H({\cal S}^{k},A)$ is connected for all $k\geq 1$  by induction.

 Clearly the statement holds for $k=1$ by the definition of skeleton.
Suppose the statement hold  for $k=n-1$, that is,
$
A\subset \bigcup_{I'\in \Sigma^{n-1}} S_{I'}(A)
$
 and $H({\cal S}^{n-1},A)$ is connected. We shall show the statement holds for $k=n$.

First, we observe that
 $$A\subset \bigcup_{i=1}^N S_i(A)\subset \bigcup_{i=1}^N S_i \left (\bigcup_{I'\in \Sigma^{n-1}} S_{I'}(A)\right )= \bigcup_{I\in \Sigma^n} S_I(A).$$

Next, we show  that $H(\mathcal{S}^n,A)$ is connected. To this end, we need only show that
for any two words
$I=i_1\dots i_n,~J=j_1\dots, j_n\in \Sigma^n$,   there is a path in $H(\cS^n, A)$ connecting $S_I$ and $S_J$.
Denote $I'=i_2\dots i_n$, $J'=j_2\dots j_n$.

\textit{Claim 1.}  The restriction of $H(\cS^n, A)$ on $\{i\}\times \Sigma^{n-1}$ is connected.

By induction hypothesis, there exist  $L_1,\dots, L_k\in \Sigma^{n-1}$
such that $S_{I'}, S_{L_1},\dots, S_{L_k}, S_{J'}$ is a path  in $H({\cal S}^{n-1}, A)$, so
$$
S_{i_1I'}, S_{i_1L_1},\dots, S_{i_1L_k}, S_{i_1J'}
$$
is a path joining $S_I$ and $S_{i_1J'}$ in $H({\cal S}^{n}, A)$.

\textit{Claim 2.} $S_I$ and $S_J$ are connected in $H(\cS^n, A)$ if $\{S_{i_1}, S_{j_1}\}$ is an edge in $H(\cS, A)$.

The assumption implies that  $S_{i_1}(A)\cap S_{j_1}(A)\neq \emptyset$, so
$$
\left (\bigcup_{I''\in \Sigma^{n-1}} S_{i_1I''}(A)\right )\cap \left ( \bigcup_{J''\in \Sigma^{n-1}} S_{j_1J''}(A)\right )\neq \emptyset,
$$
so
$
S_{i_1I''}(A)\cap S_{j_1J''}(A)\neq \emptyset
$
for some $I'',J''$ with length $n-1$; namely, there is an edge connecting $S_{i_1I''}$ and $S_{j_1J''}$.
Therefore, the path from $S_I$ to $S_{i_1I''}$ (by Claim 1), the edge
$\{S_{{i_1}I''}, S_{{j_1}J''}\}$, and the path from $S_{j_1J''}$ to $S_J$ (again by Claim 1)
form a path from $S_I$ to $S_J$.

Now we deal with the general case.
Since $H({\cal S},A)$ is connected, there exists a path connected the vertices  $S_{i_1}$ and $S_{j_1}$.
Using Claim 2 repeatedly, we obtain that $H(\cS^n, A)$ is connected.
\end{proof}

\begin{proposition}\label{single}
If a self-similar set $K$ has a skeleton $A=\{a\}$, then $K=\{a\}$.
\end{proposition}

\begin{proof}
Since the Hata graph $H(A)$ is connected,
 $S_i(a)=S_j(a)$ if there is an edge between the vertices $S_i$ and $S_j$ in  $H(A)$, so $S_1(a)=\cdots=S_N(a)$.
 This together with $A\subset \bigcup_{j=1}^N S_j(A)$ implies that $a$ is the fixed point of $S_j$
 for every $j=1,\dots, N$. It follows that  $K=\{a\}$.
   \end{proof}

\subsection{Self-similar zippers}\label{zippers}

Let $(S_j)_{j=1}^N$ be an IFS where the mappings are ordered. If there exists  a set $\{x_0,\dots, x_N\}$ of points and a sequence $(\beta_1,\dots, \beta_N)\in \{-1,1\}^N$ such that the mapping $S_j$ takes the pair $(x_0,x_N)$  into the pair $(x_{j-1}, x_j)$ if $s_j=1$, and into the pair $(x_j,x_{j-1})$ if $s_j=-1$, then we call $\{S_j\}_{j\in \Sigma}$ a  \emph{self-similar zipper}.
 $\{x_0,\dots, x_N\}$ is called the set of vertices and  call $(\beta_1,\dots, \beta_N)$ the vector of signature.
It is easy to prove that
$A=\{x_0,x_N\}$
is a skeleton of the IFS.

The definition of self-similar zipper above is introduced by \cite{Aseev03}. It is worth to mention that \cite{TetenovPurevdorj09} considers a self-similar dendrite which cannot be the attractor of any zipper. At the same time, it has a skeleton, therefore it admits a space-filling curve, as it follows from Theorem \ref{old}.

 \section{\textbf{Bifurcation pairs}}\label{sec:pair}
In this section, we give a necessary and  sufficient condition of the existence of a skeleton. Recall that $\Sigma=\{1,2,\dots,N\}$. Let $K$ be the connected self-similar set generated by an IFS $\{S_i\}_{i\in \Sigma}$.

 Denote the Hata graph of $K$ by $H=H(K)$. A sequence $(\omega_j)_{j=1}^\infty$ is called  \emph{eventually periodic}, if there exists $p\geq 1$
 such that  $(\omega_j)_{j=p}^\infty$ is periodic.

\begin{definition}
{\rm For  $e=\{S_i,S_j\}\in H$, a pair $\bomega=(\omega_k)_{k\geq 1}, \bgamma=(\gamma_k)_{k\geq 1}\in \Sigma^{\mathbb{N}}$ is  called a
 \emph{bifurcation pair} of $e$, if both sequences are eventually periodic, and $\omega_1=i, \gamma_1=j$,  $\pi(\bomega)=\pi(\bgamma)$.
 }
 \end{definition}

Let $R$ be a subgraph of the Hata graph $H$. We call $R$ a \emph{spanning graph} of $H$ if $R$ is connected, and $R$ has the same vertex set as $H$.

\begin{lemma}\label{periodic} Let $B$ be a finite subset of $K$ such that $B\subset \bigcup_{j=1}^N S_j(B)$.
Then for any point $x\in B$,  $x$ has an eventually periodic coding.
\end{lemma}

\begin{proof} Since $B$ is stable under iteration, for any $x\in B$, there exist $i_1\in \Sigma$ and $x_1\in B$ such that
$x=S_{i_1}(x_1)$. Inductively, let $i_{k+1}\in \Sigma$ and $x_{k+1}\in B$ such that $x_{k}=S_{i_{k+1}}(x_{k+1})$.
Since $B$ is finite, there must exist $m, \ell\geq 1$ such that $x_{\ell}=x_{\ell+m}$.
Hence $i_1\dots i_{\ell}(i_{\ell+1}\dots i_{\ell+m})^\infty$
is an eventually periodic coding of $x$.
\end{proof}

\begin{theorem}\label{Main: skeleton}
Let $K$ be a connected self-similar set.
Then $K$ possesses a  skeleton if and only if  there exists a spanning graph   $R\subset H(K)$ such that every
 edge $e\in R$ admits a bifurcation pair.
\end{theorem}

\begin{proof}
   First, we prove the `if' part.
Let $R$ be a spanning graph of $H(K)$ such that each edge $e\in R$ admits a bifurcation pair.
 Pick $e=\{S_i, S_j\}\in R$ and let  $\{\bomega_e, \bgamma_e\}$ be a  bifurcation pair of $e$.
 Then
 $$
\mathcal{B}_{e}=\{\sigma^k(\bomega_e);~k\in\mathbb{N} \}\cup \{\sigma^k(\bgamma_e);~k\in\mathbb{N} \},
$$
the union of orbits of $\bomega$ and $\bgamma$ under $\sigma$,  is a finite set.
We claim that
\begin{equation}\label{const-skeleton}
A=\bigcup_{e\in R}{\pi(\mathcal{B}_{e})}
\end{equation}
  is a  skeleton of $K$.

  First,  we show that the Hata graph $H(A)$ contains $R$ as a subgraph, and  hence $H(A)$ is connected since $R$ is  spanning and connected.
For $e=\{S_i,S_j\}\in R$, denote $x_e=\pi(\bomega)=\pi(\bgamma)$. Then
$$x_{e}=S_i\circ \pi (\sigma(\bomega))=S_j\circ \pi (\sigma(\bgamma))\in S_i(A)\cap S_j(A),$$
  which implies
that $e\in H(A)$. Therefore, $R$ is a subgraph of $H(A)$.

Secondly, we show that $A$ is stable. Take $x\in A$,  $x$ can be written as $x=\pi(\sigma^k(\bomega_e))$ for some $e\in R$.
Let  $\ell$ be the $(k+1)$-th entry of the sequence $\bomega_e$, then
$$x=\pi(\ell \sigma^{k+1}(\bomega_e))= S_\ell\circ \pi(\sigma^{k+1}(\bomega_e))\in S_\ell (A),$$
which verifies that $A\subset \cup_{j=1}^NS_j(A)$.

 Now we prove the `only if' part. Suppose $A$ is a skeleton of $K$. Then $R=H(A)$ is a spanning graph of $H(K)$. We claim that  every edge of $R$ admits a bifurcation pair and hence $R$ is the desired spanning graph of $H(K)$.

 Let $e=\{S_i, S_j\}$ be an edge of $R$.
 Let $z \in S_i(A)\cap S_j(A)$.  Then there exist $x,y \in A$ such that $z=S_i(x)=S_j(y)$. By Lemma \ref{periodic}, there exist eventually periodic sequences ${\boldsymbol \omega}, {\boldsymbol \gamma} \in \Sigma^{\mathbb{N}}$ such that $x=\pi({\boldsymbol \omega}), y=\pi({\boldsymbol \gamma})$.
It follows that
$i\boldsymbol \omega$ and $j{\boldsymbol \gamma}$ are  two eventually periodic codings of $z$, in other words,
$\{i\bomega,j\bgamma\}$ is a bifurcation pair. The theorem is proved.
\end{proof}

To close this section, we show that \textbf{the reptile $K$ in  Example \ref{Nonske} does not have a skeleton}.

Suppose on the contrary that $K$ has a skeleton $A$.
To guarantee that the Hata graph $H(A)$ is connected,  there must exist $i^*\in \{1,2,3\}$,
$j^*\in \{4,5,6\}$ such that $S_{i^*}(A)\cap S_{j^*}(A)\neq \emptyset$.
Take a point
$\left(\begin{matrix}
x\\
y
\end{matrix}\right)$
from the intersection. It is seen that we must have $y=1/3$.

By Theorem \ref{Main: skeleton}, the point $\left(\begin{matrix}
x\\
1/3
\end{matrix}\right)$ has two eventually codings with initial letter $i^*$ and $j^*$
respectively. It follows that
$$
\left(\begin{array}{c}
            x \\
            1/3
          \end{array}\right)
          =\sum_{k=1}^\infty 3^{-k}
          \left(\begin{array}{c}
            x_k \\
            y_k
          \end{array}\right)
          =\sum_{k=1}^\infty 3^{-k}
          \left(\begin{array}{c}
            \tilde x_k \\
            \tilde y_k
          \end{array}\right),
$$
where $\left(\begin{matrix}
x_1\\
y_1
\end{matrix}\right)=d_{i^*}$, $\left(\begin{matrix}
\tilde x_1\\
\tilde y_1
\end{matrix}\right)=d_{j^*}$,
and $\left(\begin{matrix}
x_k\\
y_k
\end{matrix}\right),\left(\begin{matrix}
\tilde x_k\\
\tilde y_k
\end{matrix}\right)\in \mathcal{D}$ for $k\geq 2$.
 Focusing on the second coordinate, we see that
 $y_1=0$, $\tilde y_1=1$ and  $y_k, \tilde y_k\in \{0,1,2\}$ for $k\geq 2$. From
$1/3=\sum_{k\geq 1} 3^{-k} y_k=\sum_{k\geq 1} 3^{-k} \tilde y_k$  we deduce that
$$y_k= 2, \ \tilde y_k= 0, \text{ for }k\geq 2.$$
 Then
$$\left(\begin{matrix}
 x_k\\
y_k
\end{matrix}\right)\in \{d_7,d_8,d_9\} \text{ and } \left(\begin{matrix}
\tilde x_k\\
\tilde y_k
\end{matrix}\right)\in \{d_1,d_2,d_3\},\quad k\geq 2,$$
 and so that $x_k\in \{\varepsilon, \varepsilon+1,\varepsilon+2\}$  and $\tilde x_k\in \{0,1,2\}$ for $k\geq 2$.
Since $(x_k)_{k\geq 1}$ and $(\tilde x_k)_{k\geq 1}$ are eventually periodic, we have that
$\sum 3^{-k}x_k$ is irrational and $\sum 3^{-k} \tilde x_k$ is rational, which contradicts that
$\sum 3^{-k}x_k=\sum 3^{-k} \tilde x_k$. This contradiction completes the proof.

\section{\textbf{Finite type condition }}\label{Sec:finitetype}

In this section, we deal with self-similar sets of  finite type, and then Theorem \ref{Main:res} can be proved.

\subsection{Terminology of Graphs} First, we recall some terminologies of graph theory, see for instance, \cite{BalakrishnanRanganathan2000}.
Let $H$ be a directed graph. We shall use $\{e_1,\dots,e_k\}$ to denote a walk consisting of
the edges $e_1,\dots, e_k$. (For the definition of walk see \cite[Definition 1.5.1]{BalakrishnanRanganathan2000}.) We call the starting vertex and terminate vertex of a walk the \emph{origin} and \emph{terminus}, respectively.
The walk is \emph{closed} if the origin of $e_1$ and the terminus of $e_k$ coincide.
A walk is called a \emph{trail}, if all the edges appearing in the walk are distinct. A trail  is called a \emph{path} if all the vertices are distinct. A closed path is called a \emph{cycle}.

\subsection{Neighbor maps}
Recall that $\Sigma=\{1,\dots,N\}$ and $\Sigma^\ast=\bigcup_{m=0}^\infty{\Sigma^m}$. For $I, J\in \Sigma^*$, we say that
$I$ and $J$ are \emph{non-comparable}, if neither $I$ is a prefix of $J$, nor $J$ is a prefix of $I$.
We use $|I|$ to denote the length of the word $I$, and use $I^-$ to denote the word obtained by deleting the last letter of $I$.

Let $\{S_i\}_{i\in\Sigma}$ be an IFS with invariant set $K$. Without loss of generality, we assume that $0\in K$. Denote the contraction ratio of $S_j$ by $r_j$, and denote $r_I$ the contraction ratio of the map $S_I$. Set
$$r_*=\min_{j\in \Sigma} |r_j|,\quad r^*=\max_{j\in \Sigma} |r_j| .$$

\begin{definition}For a non-comparable pair $I, J\in \Sigma^*$, the map
$$f(z)=S_J^{-1}\circ S_I(z)$$
 is called a \emph{neighbor map}, and we call it a \emph{feasible neighbor map} if
$$r_*< (r_J)^{-1}r_{I}\leq (r_*)^{-1}.$$
\end{definition}

  A feasible neighbor map describes
the relative position of two cylinders which have about the same size.
 Let ${\mathcal N}_0$
 be the collection of feasible neighbor maps.
 For any $i,j\in \Sigma$ with $i\neq j$, we call $S_j^{-1}\circ S_i$ the \emph{basic neighbor maps}.
Clearly all basic neighbor maps are feasible.

\subsection{Neighbor graph}
The neighbor graph  is a directed graph with vertex set ${\mathcal N}_0$.
Let $\epsilon$ be the empty word, and we set $S_\epsilon=id$ to be the identity map for convention.
 For $f,g\in {\mathcal N}_0$,
we say that there is an edge from $f$ to $g$, if there exist $i\in \Sigma$ such that
$$
f\circ S_i=g \text{ or } S_i^{-1}\circ f=g;
$$
in the first case, we denote the edge by $(\epsilon,i)$, or $(f,\epsilon,i,g)$,
while in the second case, we  name the edge by $(i, \epsilon)$, or $(f,i, \epsilon,g)$.
We call the graph defined above the \emph{complete neighbor graph},
and denote it by $\Delta_0$.

\begin{definition}\label{def-neighbor}
{\rm An IFS $\{S_j\}_{j\in \Sigma}$ is said to  satisfy \emph{the finite type condition} if
$$
{\mathcal N}=\left \{f\in {\mathcal N}_0;~f(K)\cap K\neq \emptyset \right \}
$$
is a finite set. We call the restriction of $\Delta_0$ on ${\cal N}$, denoted by $\Delta$, the \emph{neighbor graph}.
}
\end{definition}

The following lemma is folklore. We give a proof for the sake of readers.

\begin{lemma}\label{basic-type}
(i) For any $f\in {\mathcal N}_0$, there is a walk in $\Delta_0$ starting at some basic neighbor map and terminating at $f$. (If $f$ is a basic neighbor map, then it is an empty walk starting from itself. )

    (ii)  For any  $f\in {\cal N}$,  there is at least one edge in $\Delta$ emanating from $f$.
\end{lemma}
\begin{proof} (i) Suppose $f=S_J^{-1}\circ S_I\in {\mathcal N}_0$ is a feasible neighbor map, we prove the lemma by induction on $|I|+|J|$. The lemma
is clearly true if $|I|+|J|=2$, since $f$ itself is a basic neighbor map. Let $r$ be  the contraction ratio of $f$.

If $|r|> 1$, then
 $h=S_{J^-}^{-1}\circ S_I$
is clearly a feasible neighbor map,
 and there is an edge $(h,j^*, \epsilon, f)$ from $h$ to $f$ where $j^*$ is the last letter of $J$. By induction, there is a walk starting at some basic neighbor map
and terminating at $h$, and this walk can be extended to $f$.

If $|r|\leq 1$,
let $h=S_{J}^{-1}\circ S_{I^-}$, and the above discussion still holds.

(ii) Without loss of generality, let us assume that the contraction ratio of $f$ is no less than $1$.
Pick $x\in f(K)\cap K$,
then there exists $j\in \Sigma$ such that $x\in f\circ S_j(K)$.  Clearly
$f\circ S_j\in {\cal N}$ and there is an edge $(\epsilon, S_j)$ from $f$ to $f\circ S_j$.
\end{proof}

\subsection{Searching bifurcation pairs with the neighbor graph}
The existence of bifurcation pair can be characterized in terms of the neighbor graph.

\begin{lemma}\label{f-loop}
An edge $e=\{S_i, S_j\}$ of the Hata graph $H(K)$ admits a bifurcation pair if  there is an eventually periodic walk on $\Delta$ emanating from the basic neighbor map $S_i^{-1}\circ S_j$ or $S_j^{-1}\circ S_i$.
\end{lemma}

\begin{proof} Let $(\omega_n)_{n\geq 1}$ be an eventually periodic walk emanating from $S_i^{-1}\circ S_j$.
This means that there exist $I_0, I_1, J_0, J_1\in \Sigma^*$ such that
$S_{iI_0I_1^k}^{-1}S_{jJ_0J_1^k}\in {\cal N}$ for all $k\geq 1$, so
\begin{equation}\label{eq:ccode}
S_{iI_0I_1^k}(K)\cap S_{jJ_0J_1^k}(K)\neq \emptyset
\end{equation}
for all $k\geq 1$. Let $x$ be the point with the coding $iI_0(I_1)^\infty$, then by
\eqref{eq:ccode}, $x$ has a second coding
$jJ_0(J_1)^\infty$. Therefore,  $\{iI_0(I_1)^\infty, jJ_0(J_1)^\infty\}$ is a bifurcation pair of $e=\{S_i, S_j\}$.
\end{proof}

As an immediate consequence of the above lemma, we have

\begin{theorem}\label{Main: finite-Type}
 Let $K$ be a self-similar set satisfying the finite type condition.
Let $H=H(K)$ be the Hata graph. Then every edge $e\in H$ admits a bifurcation pair.
\end{theorem}

 \begin{proof}  Suppose $e=(S_i, S_j)\in H(K)$,
then $S_j(K)\cap S_i(K)\neq \emptyset$. Thus $f=S_j^{-1}\circ S_i\in {\mathcal N}$.

By Lemma \ref{basic-type} (ii), there is an infinite walk in $\Delta$ starting from $f$;
moreover, this walk must be eventually periodic since $\Delta$ is a finite graph.
Finally, by Lemma \ref{f-loop}, $e$ admits a bifurcation pair.
\end{proof}

Another consequence of Lemma \ref{f-loop} is the following.

 \begin{corollary}\label{Main:New}
Let $K$ be a connected self-similar set.
If  there exists a spanning graph $R\subset H(K)$ such that, for every
 edge $e=\{S_i, S_j\}\in R$,  there is an eventually periodic walk on $\Delta$ emanating from the  neighbor map $S_i^{-1}\circ S_j$ or $S_j^{-1}\circ S_i$, then $K$ possesses a  skeleton.
\end{corollary}

%\begin{proof}
%It is  obvious by Theorem \ref{Main: skeleton} and Lemma \ref{f-loop}.
%\end{proof}

\section{\textbf{Algorithm  and Examples}}\label{sec:skeleton}

In this section, we formulate an algorithm to obtain skeletons for a self-similar set of finite type.
Besides, we will give three examples related to the so-called single-matrix IFS.

\subsection{Algorithm}
Let $\mathcal S=\{S_j\}_{j\in {\Sigma}}$ be an IFS such that the invariant set $K$ is connected and $\mathcal{S}$ satisfies the finite type condition.
Summarizing the results in the previous sections, we give an  algorithm consisting of
the following four steps:

{\it
\ \ $(i)$ Compute the  neighbor graph $\Delta$;

\ $(ii)$ Compute the Hata graph $H(K)$, and choose a spanning graph $R$;

$(iii)$ For each  $e\in R$, find a bifurcation pair of $e$ by Theorem \ref{Main: finite-Type};

$(iv)$ Construct a skeleton according to \eqref{const-skeleton}.
}

\subsection{Neighbor graph of IFS with uniform contraction ratio}
If all the contraction ratios $r_j$ of the similitudes in an IFS have the same value, then the neighbor graph can be simplified as follows (see \cite{BandtMesing09}).
First, the sets ${\mathcal N}_0$ and ${\cal N}$ can be reduced to
$$\mathcal{N}_0=\bigcup_{n=1}^\infty \left \{S_{J}^{-1}S_{I};~~{I},{J} \in \Sigma^n, i_1\neq j_1\right \},$$
$$
{\cal N}=\{f\in {\cal N}_0;~ f(K)\cap K\neq \emptyset\}.
$$
Secondly,  the neighbor graph can be  simplified as follows:
For $f, g\in {\cal N}_0$, there is an edge from $f$ to $g$, if there exists a pair $i,j \in \Sigma$ such that
$$
S_j^{-1}\circ f\circ S_i=g;
$$
in this case, we name the edge  by $(f,j,i,g)$, or label the edge by $(j,i)$ in short.

  \subsection {Single-matrix IFS}  A single-matrix IFS is a special type of IFS with uniform ratio giving by
 \begin{equation}\label{digit IFS}
 S_i(z)=rM(z+d_i), \quad i=1,\dots,N,
 \end{equation}
 where $0<r<1$ and $M$ is an orthogonal matrix, see \cite{LuoYang2010}.
Let us denote ${\mathcal D}=\{d_1,\dots, d_N\}$, and define
$$
{\mathcal D}^*=\bigcup_{n=1}^\infty \left \{\sum_{k=0}^{n-1} (rM)^{-k}x_k;~x_k\in D-D\right \}.
$$
By induction, one can show that
$$
{\mathcal N}_0=\{z\mapsto z+b;~b\in  {\mathcal D}^*\}.
$$
In particular, if ${\mathcal D}^*$ is uniformly discrete, that is, $\inf_{x,y\in {\cal D}^*} |x-y|>0$,  then

  \emph{the set $\{|b|<T;~b\in {\mathcal D}^*\}$ is a finite set for any $T>0$,}\\
and hence the IFS satisfies the finite type condition.

\subsection{Examples}
\begin{example}
\textbf{Skeleton of Terdragon.}
 {\rm  Recall the  IFS of Terdragon is
$$S_1(z)=\lambda z+1,~S_2(z)=\lambda z+\omega,S_3(z)=\lambda z+\omega^2,$$
where $\lambda=\exp(\pi \mi/6)/\sqrt{3}$ and $\omega=\exp(2\pi \mi/3)$.
 The IFS is of the form \eqref{digit IFS} with $rM=\exp(\mi \pi/6)/\sqrt{3}$, and
$${\mathcal D}=\sqrt{3}\exp(-\mi \pi/6)\cdot \{1,\omega,\omega^2\}.$$

Since $(rM)^{-1}=1-\omega$, it is easy to show that
 ${\mathcal D}^*$ is a subset of the lattice $\Z+\Z\omega$ and hence it is uniformly discrete.

$(i)$ The feasible neighbor map set is
$$
{\mathcal N}=\{f_1, f_2, f_3, f_4, f_5, f_6\}:=\{S_2^{-1}\circ S_3, S_1^{-1}\circ S_3,
S_1^{-1}\circ S_2,  S_2^{-1}\circ S_1, S_3^{-1}\circ S_1, S_3^{-1}\circ S_2\}.
$$
The neighbor graph $\Delta$ (see Figure \ref{neigh-terdragon}) is
$$
f_1 \stackrel{(3,1)}{\longrightarrow} f_1, \
f_1\stackrel{(3,2)}{\longrightarrow}f_2, \
f_2\stackrel{(2,1)}{\longrightarrow}f_2, \
f_2\stackrel{(3,1)}{\longrightarrow}f_3;
$$
$$
f_3 \stackrel{(2,3)}{\longrightarrow} f_3, \
f_3\stackrel{(2,1)}{\longrightarrow}f_6, \
f_6\stackrel{(1,3)}{\longrightarrow}f_6, \
f_6\stackrel{(2,3)}{\longrightarrow}f_5;
$$
$$
f_5 \stackrel{(1,2)}{\longrightarrow} f_5, \
f_5\stackrel{(1,3)}{\longrightarrow}f_4, \
f_4\stackrel{(3,2)}{\longrightarrow}f_4, \
f_4\stackrel{(1,2)}{\longrightarrow}f_1.
$$

\begin{figure}[h]
  \subfigure[]{\includegraphics[width=0.4\textwidth]{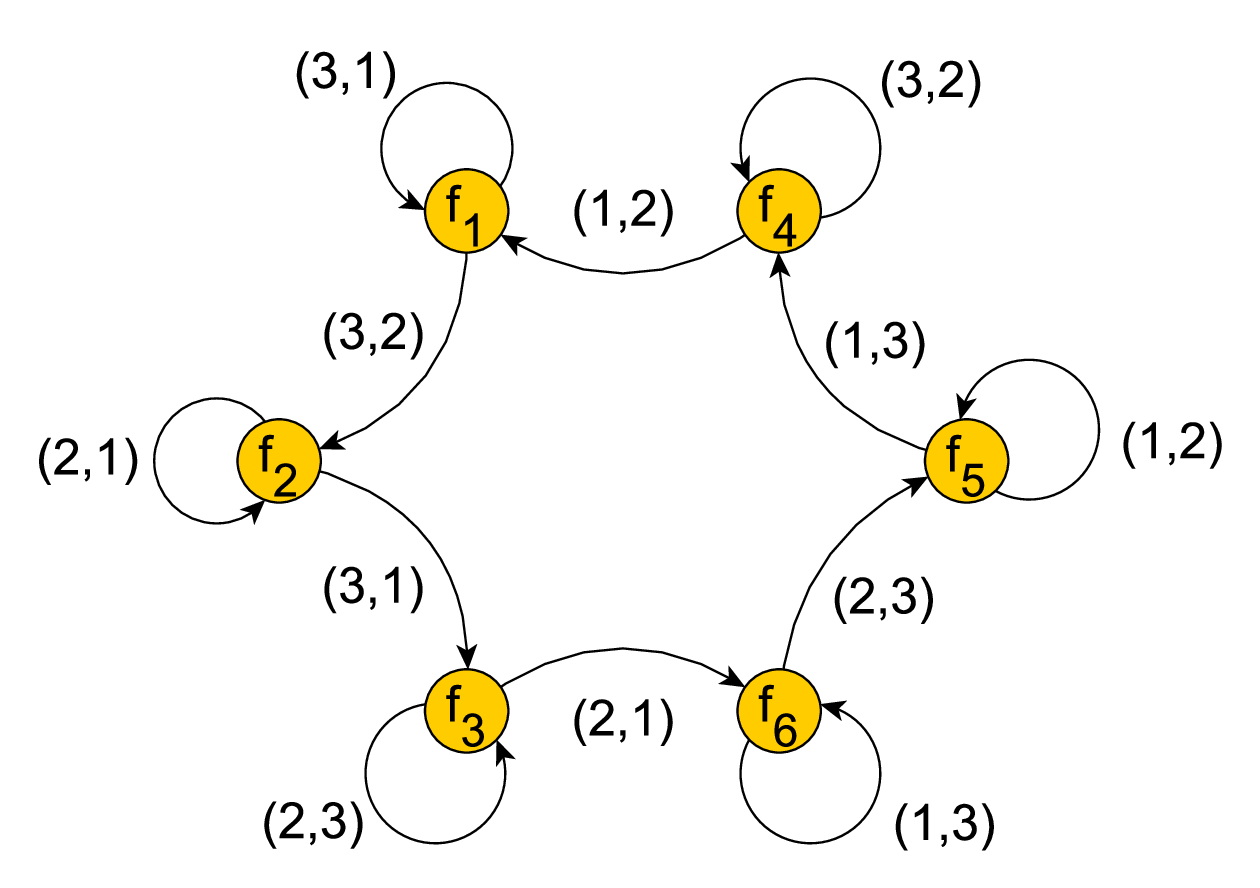}}\quad
   \subfigure[]{\includegraphics[width=0.3 \textwidth]{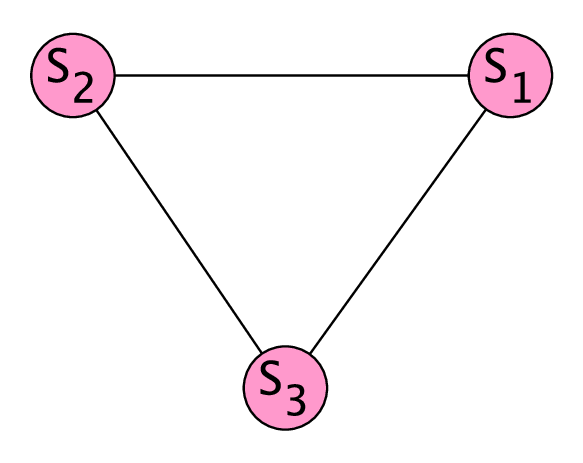}}\\
  \caption{(a): Neighbor graph of Terdragon. (b): Hata graph $H$ of Terdragon.}\label{neigh-terdragon}
\end{figure}

$(ii)$ We choose $R=H$.

(1) The first  skeleton (see Figure \ref{SFC-terdragon} (a)).
For the three edges $\{S_1, S_2\}$, $\{S_2, S_3\}$ and $\{S_3, S_1\}$
of $R$, we choose the self-loops of $f_3=S_1^{-1}\circ S_2, f_1=S_2^{-1}\circ S_3, f_5=S_3^{-1}\circ S_1$ respectively, then we obtain the associated bifurcation pairs
$\{12^\infty, \ 23^\infty\}, \  \{23^\infty ,\ 31^\infty\}, \ \{31^\infty,  12^\infty\}$ respectively.  Notice that $\pi(12^\infty)=\pi(23^\infty)=\pi(31^\infty)$=0.
So, we obtain the following skeleton
 $$\{a_1,a_2,a_3\}=\pi\{1^\infty, 2^\infty, 3^\infty\}=\{S_3^{-1}(0), S_1^{-1}(0), S_2^{-1}(0)\}=\{-\omega^2,-1,-\omega\}/\lambda.$$
This skeleton is used in \cite{Dekking82b} and \cite{DaiRaoZhang15}.

\begin{figure}[h]
  \centering
  \includegraphics[width=0.35 \textwidth]{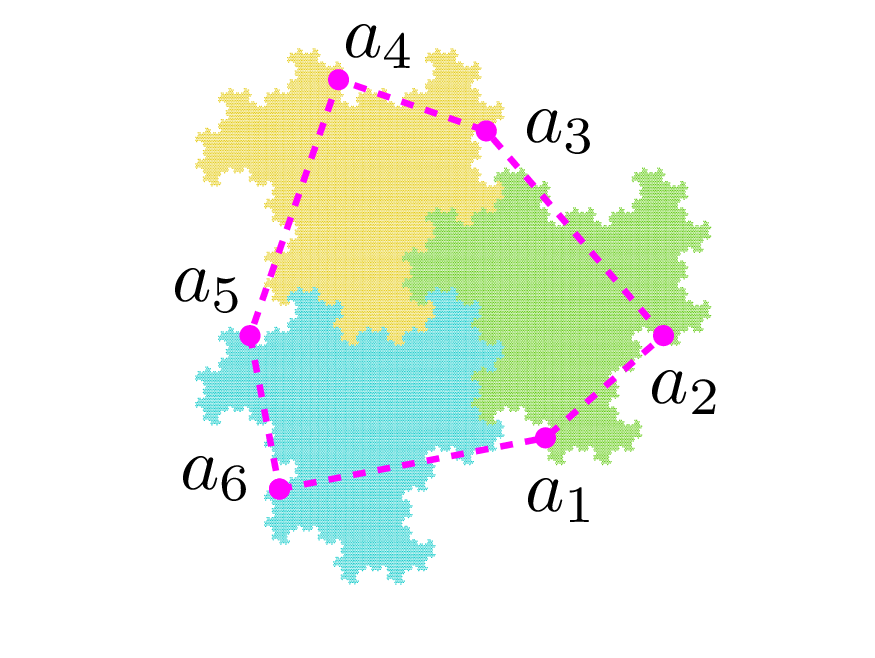}
  \caption{Another skeleton of Terdragon}
  \label{new-skeleton}
\end{figure}

(2) The second skeleton (see Figure \ref{new-skeleton}).  Notice that
   $$\big((3,2)(3,1)(2,1)(2,3)(1,3)(1,2)\big)^\infty$$
    is a infinite repetition of a cycle in $\Delta$ passing all the vertices.
     For $\{S_2,~S_3\}\in R$,   we  regard  $f_1=S_2^{-1}\circ S_3$ as the starting point of the above cycle.
     Then the bifurcation pair associated to $\{S_2,~S_3\}$ is $\{2(332211)^\infty, 3(211332)^\infty\}$.
   Similarly, the other two bifurcation pairs associated to $\{S_1,~S_2\}$ and $\{S_3,~S_1\}$ are
$ \{1(221133)^\infty,  2(133221)^\infty\}$ and $\{3(113322)^\infty, 1(322113)^\infty\}$, respectively.
Denote $\bomega=(332211)^\infty$, then
the resulting skeleton is (see Figure \ref{new-skeleton})
$$
 \begin{array}{rl}
 &\{a_1,a_2,a_3,a_4, a_5, a_6\}=\pi\{\sigma^k(\bomega); k=0,\dots,5\}\\
=&\{\frac{6}{7}-\frac{4\sqrt{3}}{7}\mi, \frac{12}{7}-\frac{\sqrt{3}}{7}\mi, \frac{3}{7}+\frac{5\sqrt{3}}{7}\mi,-\frac{9}{14}+\frac{13\sqrt{3}}{14}\mi, -\frac{9}{7}-\frac{\sqrt{3}}{7}\mi,-\frac{15}{14}-\frac{11\sqrt{3}}{14}\mi \}.
 \end{array}
 $$
\begin{figure}[h]
  \includegraphics[width=0.5 \textwidth]{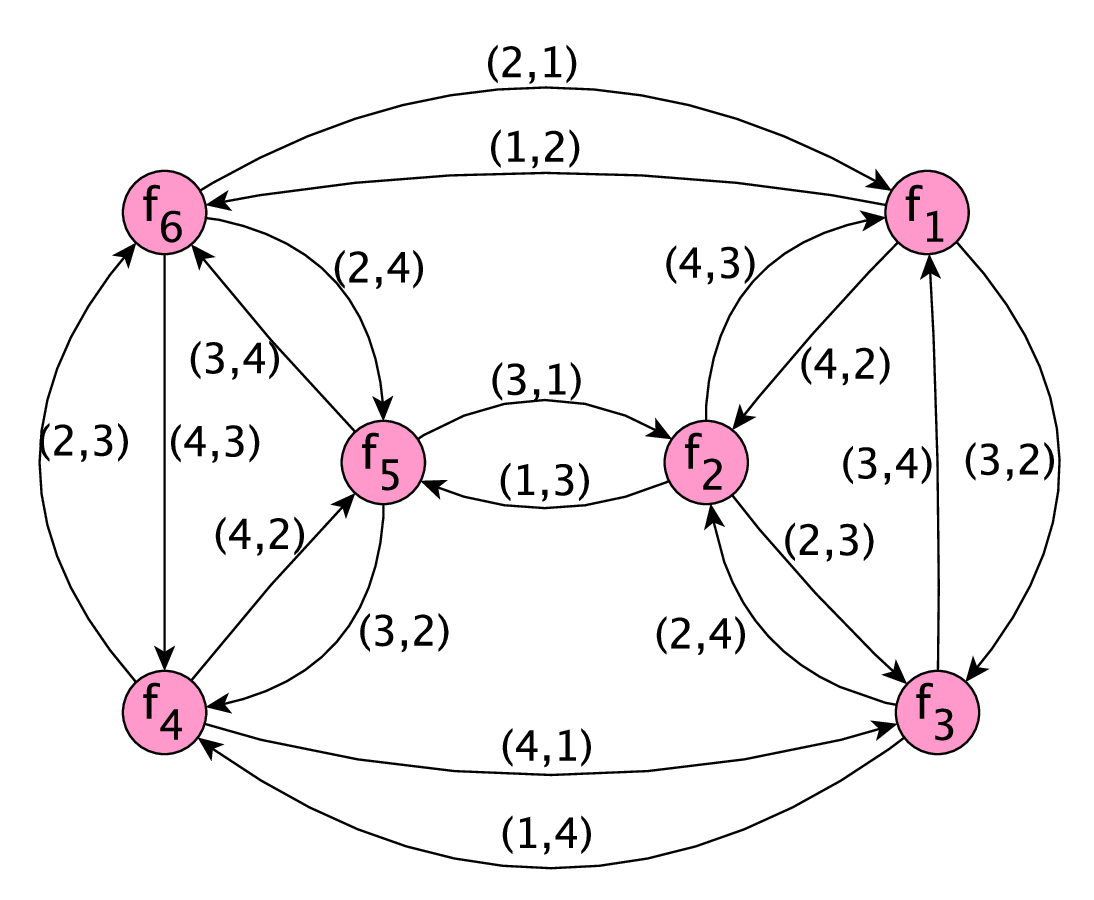}\\
  \caption{Neighbor graph of Four tile star  restricted on basic neighbor maps.}\label{Neigh-Four-star}
\end{figure}
}

\end{example}

\begin{example} \textbf{Skeleton of the four-tile star}. {\rm  The associated IFS is
$$S_1(x)=-\frac{1}{2}x,\ S_2(x)=-\frac{1}{2}x-\mi,\ S_3(x)=-\frac{1}{2}x+\exp(\frac{5\pi\mi}{6}),\
S_4(x)=-\frac{1}{2}x+\exp(\frac{\pi\mi}{6}).$$
By the same argument of the previous example, the four-tile star also satisfies the  finite type condition.

$(i)$ Figure \ref{Neigh-Four-star}  illustrates a subgraph of the neighbor graph restricted on the six basic neighbor maps $f_1,\dots, f_6$ given by
$$f_1=f_6^{-1}=S_1^{-1}\circ S_2, \ f_2=f_5^{-1}=S_1^{-1}\circ S_3,\ f_3=f_{4}^{-1}=S_1^{-1}\circ S_4.$$

$(ii)$ The Hata graph $H$ is depicted in Figure \ref{four-tile-R} (b), and the spanning graph we choose is
$$
R=\{\{S_1,S_2\}, \{S_1,S_3\}, \{S_1, S_4\}\}.
$$

\begin{figure}[h]
  \centering
  \subfigure[Four-tile star]{\includegraphics[width=0.32 \textwidth]{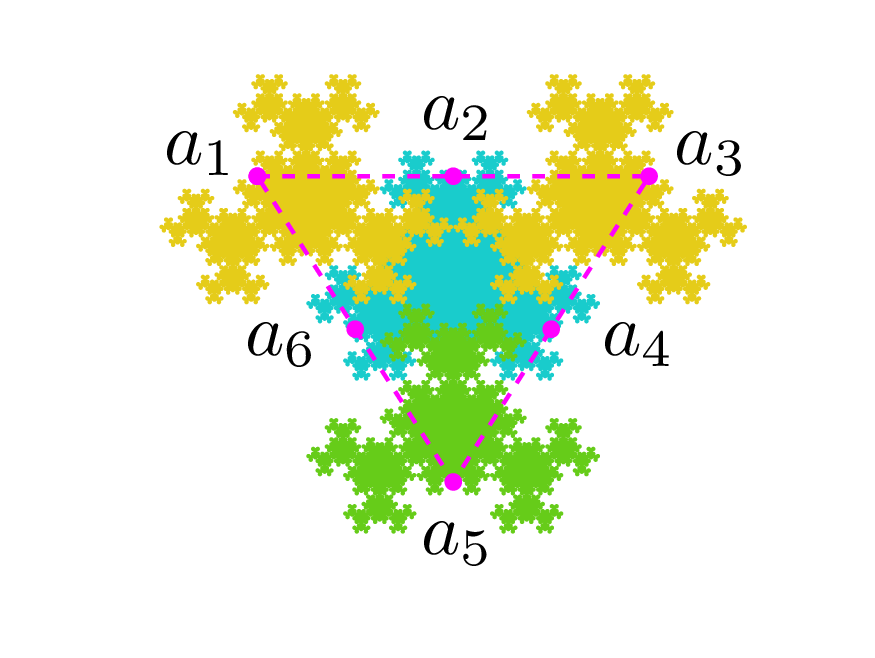}}
  \subfigure[Hata graph $H$]{\includegraphics[width=.25 \textwidth]{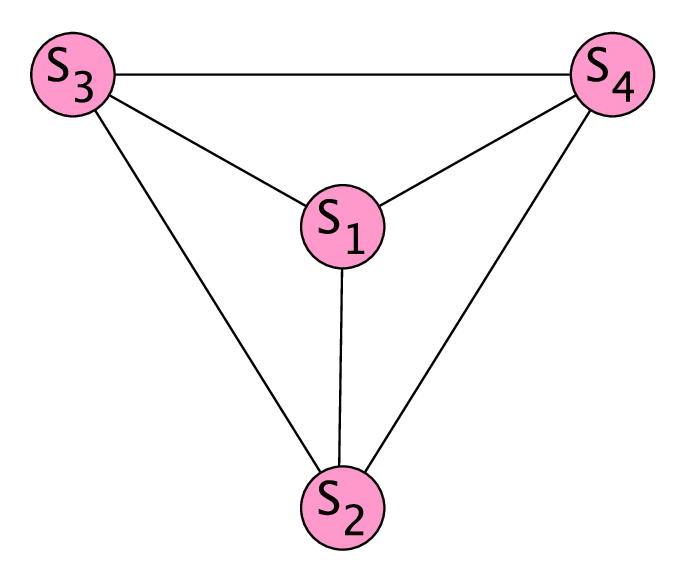}}
 \subfigure[A spanning graph $R$]{\includegraphics[width=.25 \textwidth]{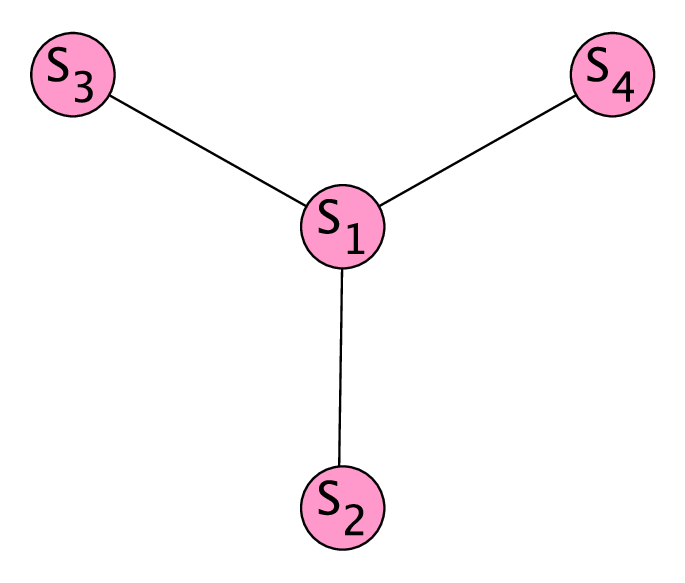}}
  \caption{   Four-tile star. }
  \label{four-tile-R}
\end{figure}

$(iii)$  To find a bifurcation pair of $\{S_1,S_2\}$, we need an eventually walk starting
from $f_1=S_1^{-1}\circ S_2$. We choose $\big((1,2), (2,1)\big)^\infty$ which is the infinite repetition of a cycle. Then we have the bifurcation pairs $\{1(12)^{\infty}, 2(21)^{\infty}\}$.  Similarly, we get the  bifurcation pairs of the other two edges $\{S_1,S_3\}$ and $\{S_1,S_4\}$, which are
$ \{1(13)^\infty, 3(31)^\infty \}, \ \{1(14)^\infty, 4(41)^\infty\},
$
respectively.

$(iv)$ The resulted skeleton is
 $$
 \begin{array}{rl}
 \{a_1,a_2,a_3,a_4, a_5, a_6\} &=\pi\{(31)^\infty, (12)^\infty, (41)^\infty, (13)^\infty, (21)^\infty, (14)^\infty\}\\
 &=\{\frac{-2\sqrt{3}+2\mi}{3}, \frac{2\mi}{3}, \frac{2\sqrt{3}+2\mi}{3}, \frac{\sqrt{3}-\mi}{3}, \frac{-4\mi}{3},
 \frac{-\sqrt{3}-\mi}{3}\}.
 \end{array}
 $$
 (See Figure \ref{four-tile-R} (a).)

}
\end{example}

\begin{example}
\textbf{Space-filling curves of Sierpi\'nski carpet.}  Here we  display the SFCs constructed from different skeletons of Sierpi\'nski carpet by using the postive Euler-tour method in \cite[Section 5]{DaiRaoZhang15}.
\begin{itemize}
\item[(1)] We choose the four vertices $\{b_1,b_2,b_3,b_4\}$ as a skeleton (Figure \ref{Hata_gra1} (a)),
then we get the space-filling  curve as Figure \ref{CarpetSFC-1}.

\begin{figure}[htpb]
\includegraphics[width=0.28 \textwidth]{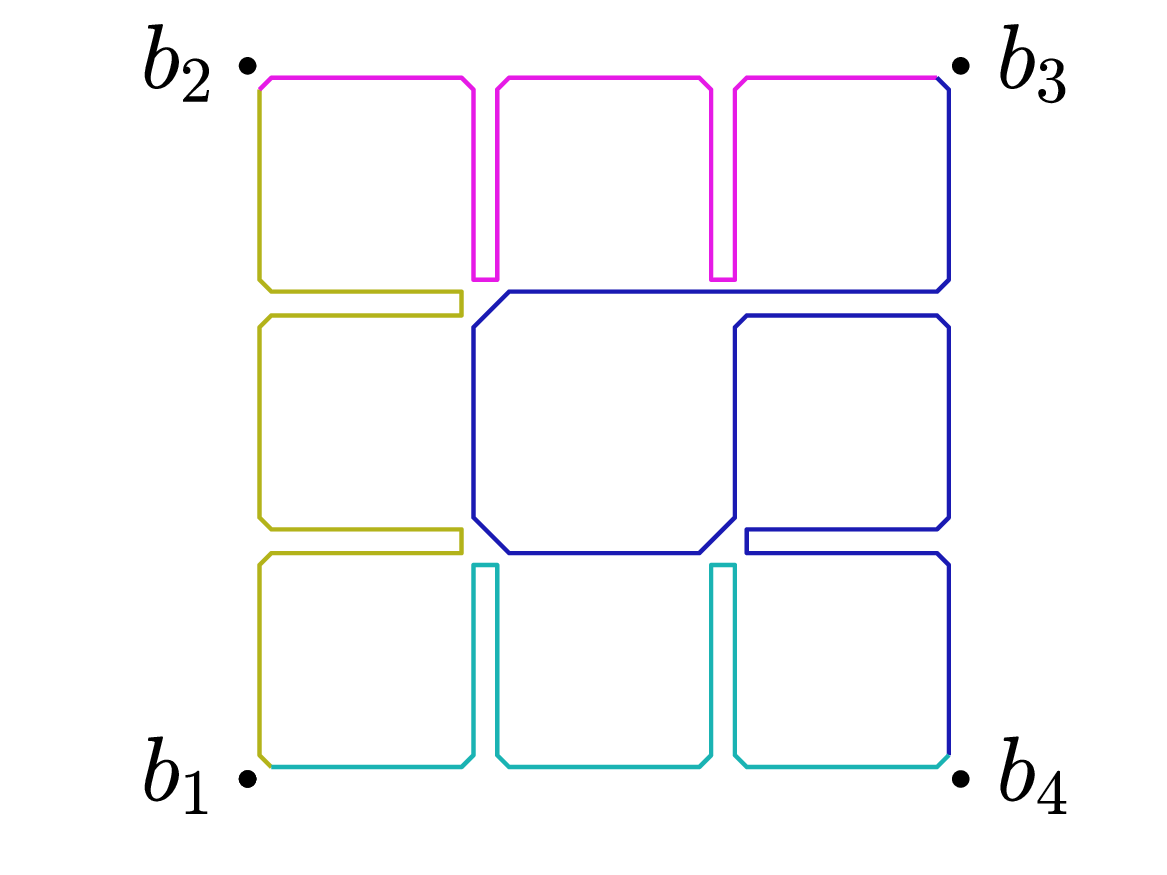}\includegraphics[width=0.28 \textwidth]{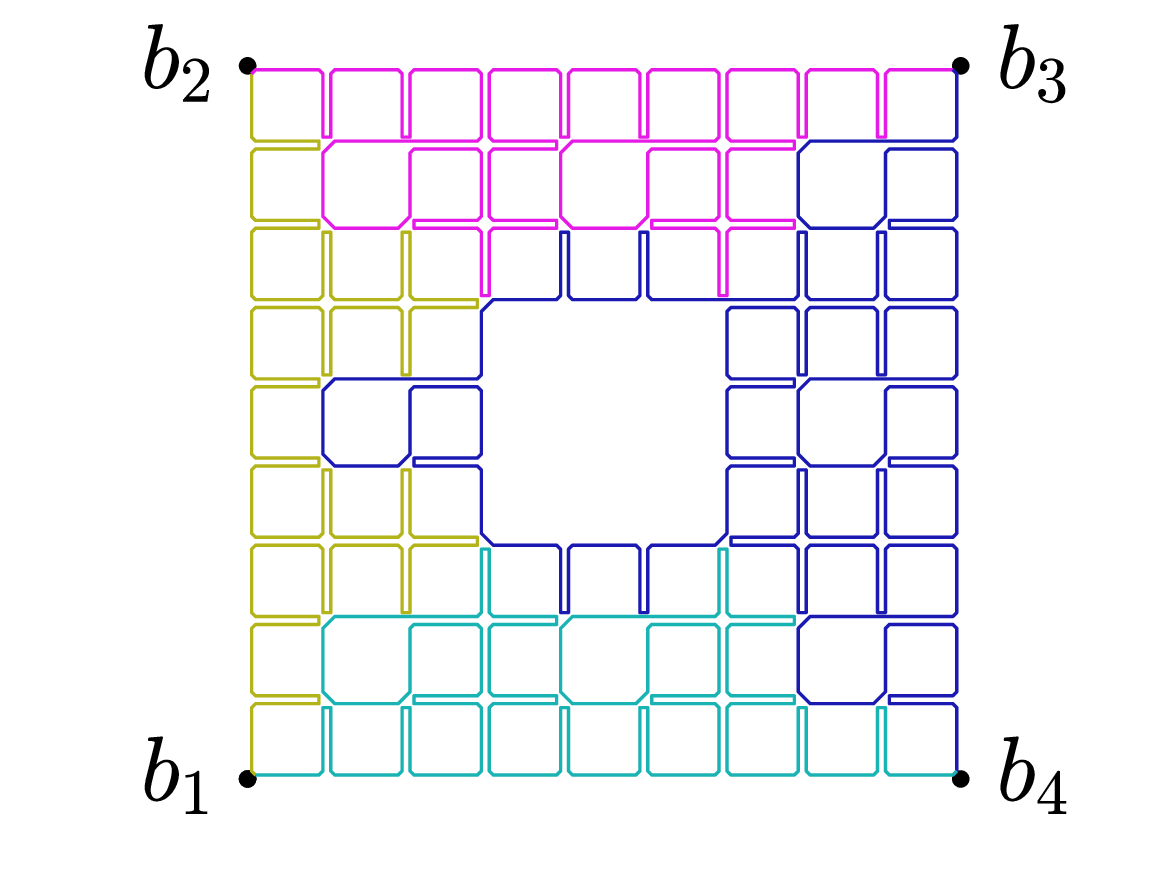}\includegraphics[width=0.28 \textwidth]{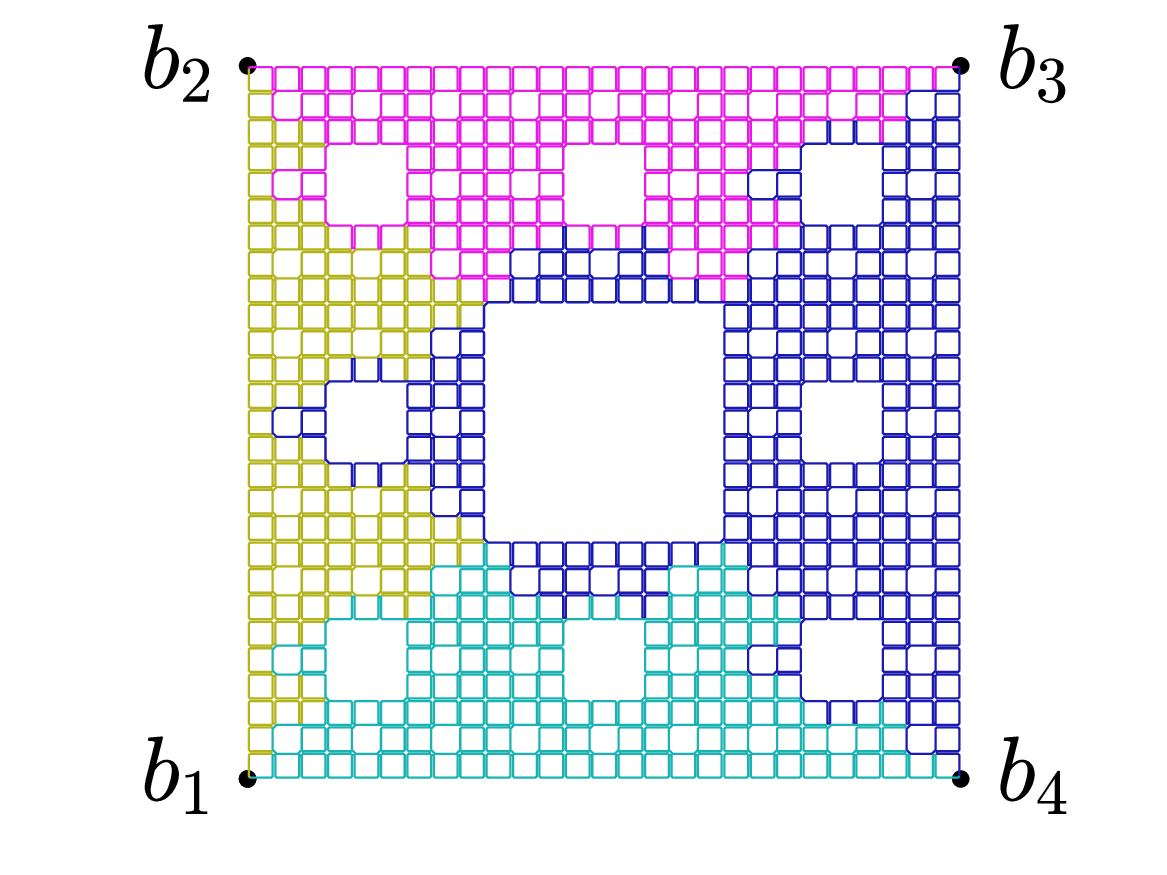}
\caption{The first three approximations of SFC of Sierpi\'nski carpet with four vertice as a skeleton.}\label{CarpetSFC-1}
\end{figure}

\item[(2)] We choose three vertice $\{b_1, b_2, b_4\}$ as a skeleton (Figure \ref{Hata_gra1} (a)). Figure \ref{CarpetSFC-2}   gives the first three approximations of the SFC.

\begin{figure}[htpb]
\includegraphics[width=0.28 \textwidth]{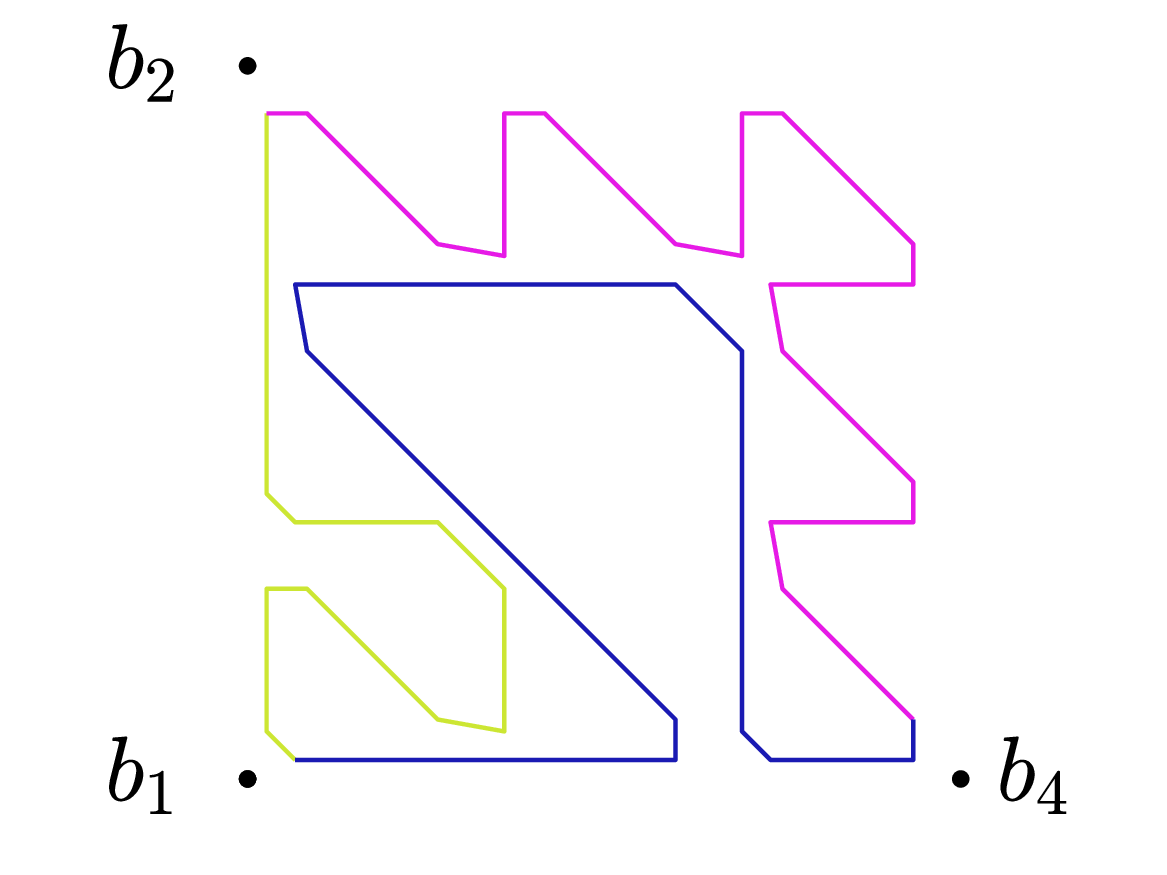}\includegraphics[width=0.28 \textwidth]{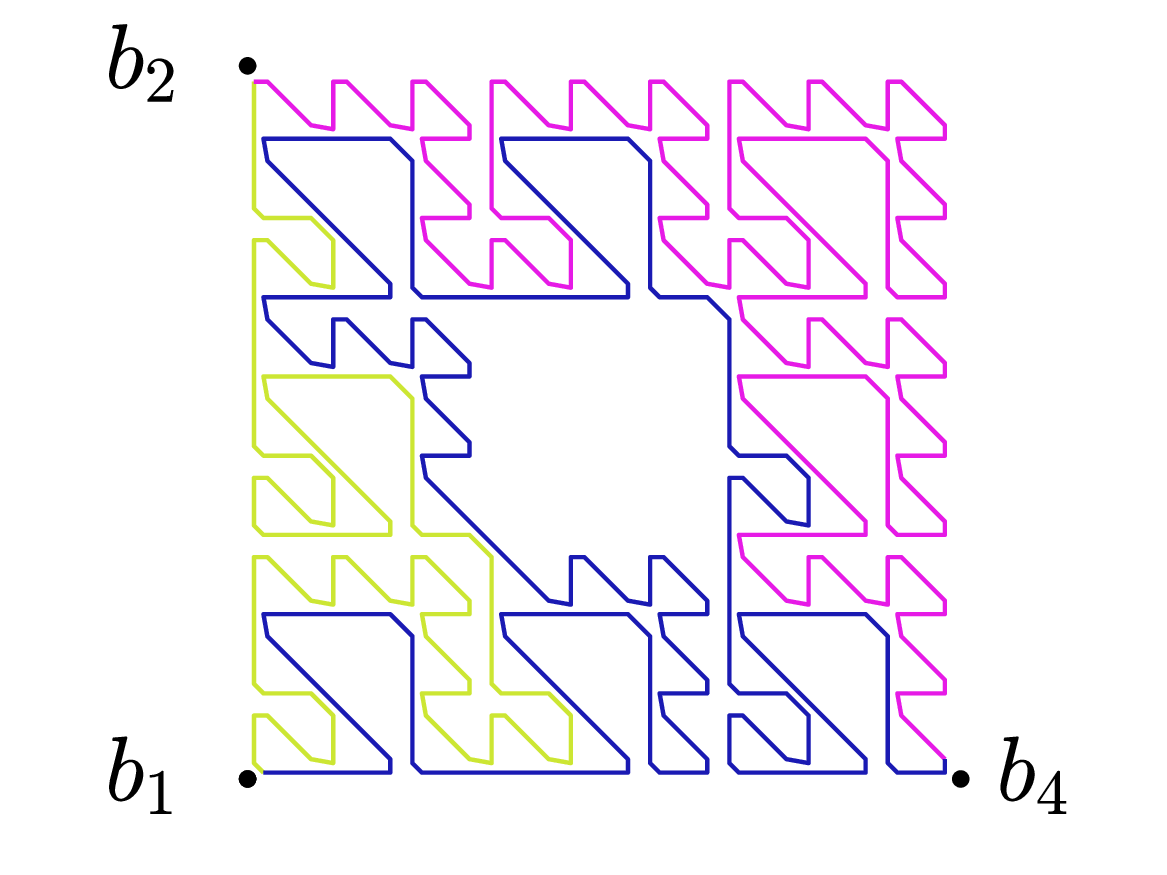}\includegraphics[width=0.28 \textwidth]{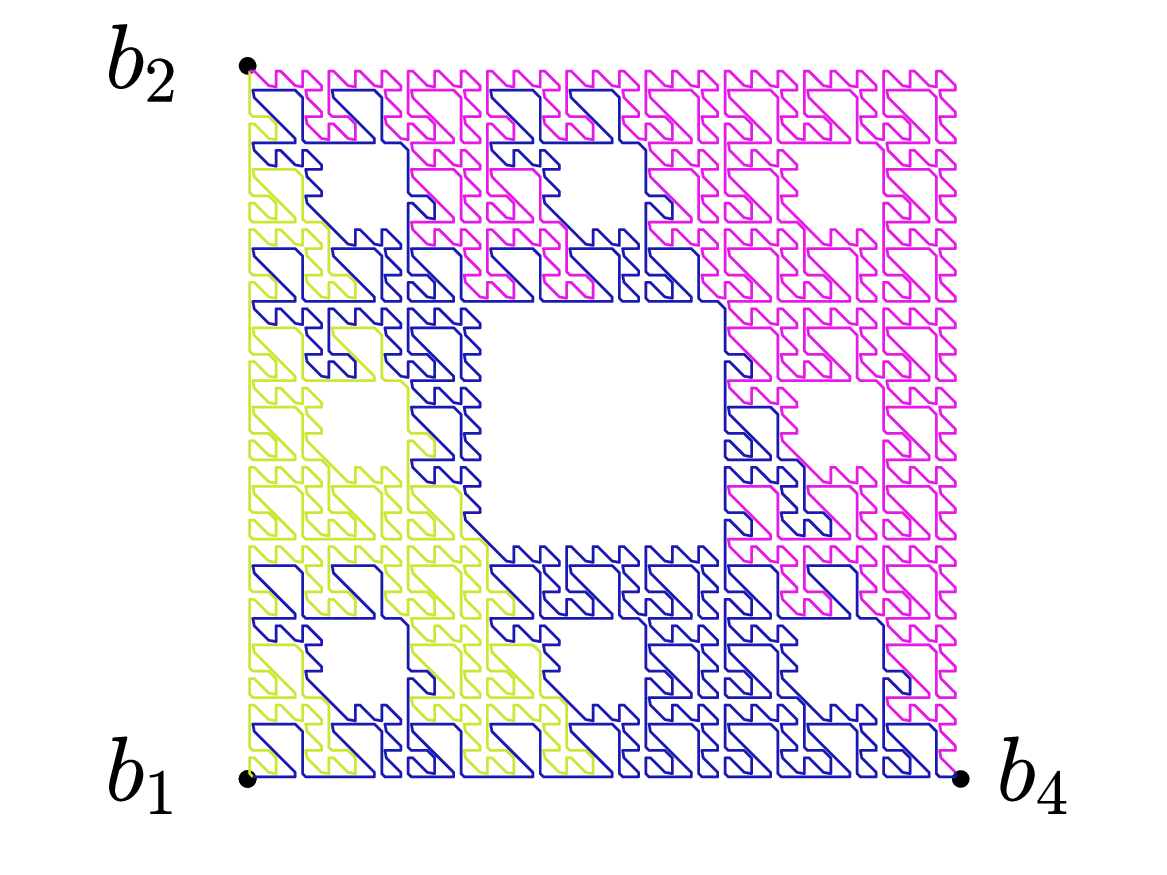}
\caption{The first three approximations of SFC of Sierpi\'nski carpet with three vertices as a skeleton.}\label{CarpetSFC-2}
\end{figure}

\item[(3)] Consider the skeleton of four middle points $\{a_1,a_2,a_3,a_4\}$ (Figure \ref{Hata_gra1} (a)). The according approximation curves are shown in Figure \ref{CarpetSFC-3}. A  detailed discussion of this SFC can be found in \cite[Example 5.2]{DaiRaoZhang15}.

\begin{figure}[htpb]
\includegraphics[width=0.28\textwidth]{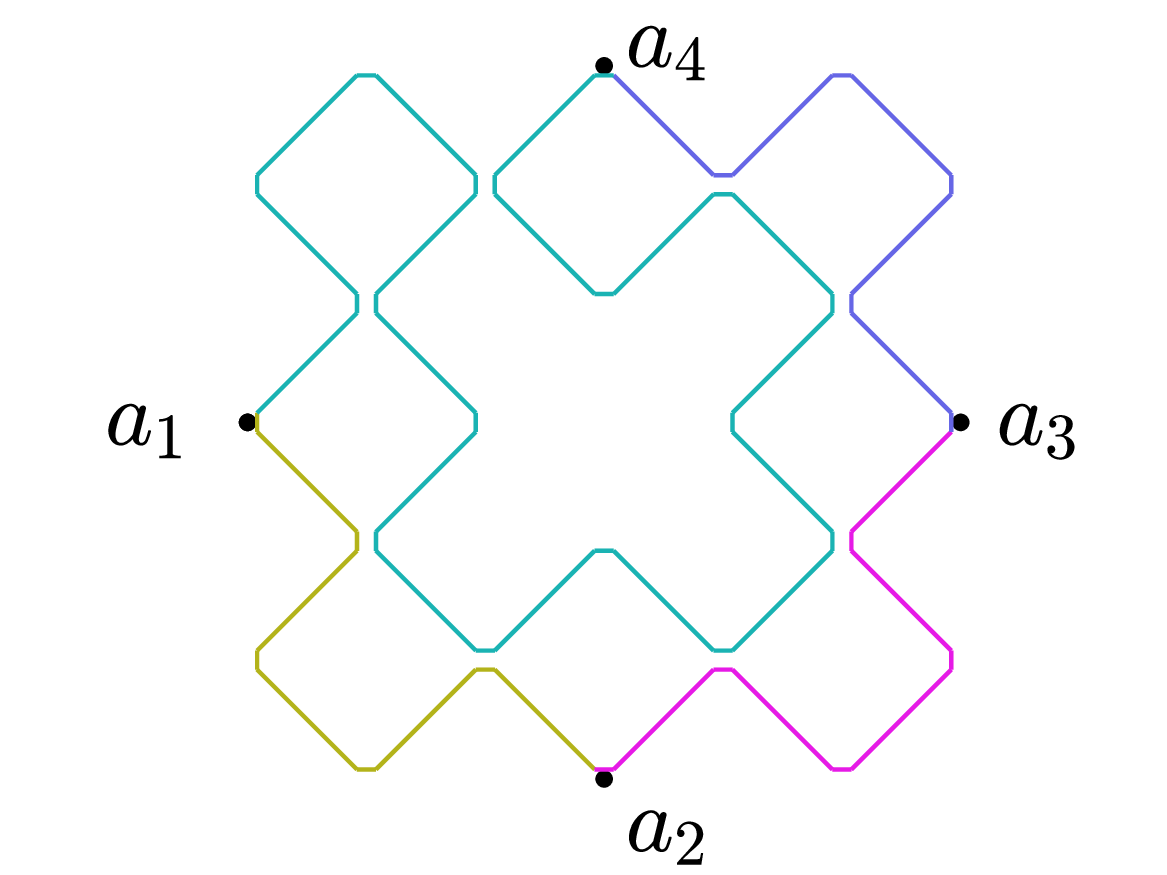}
\includegraphics[width=0.28\textwidth]{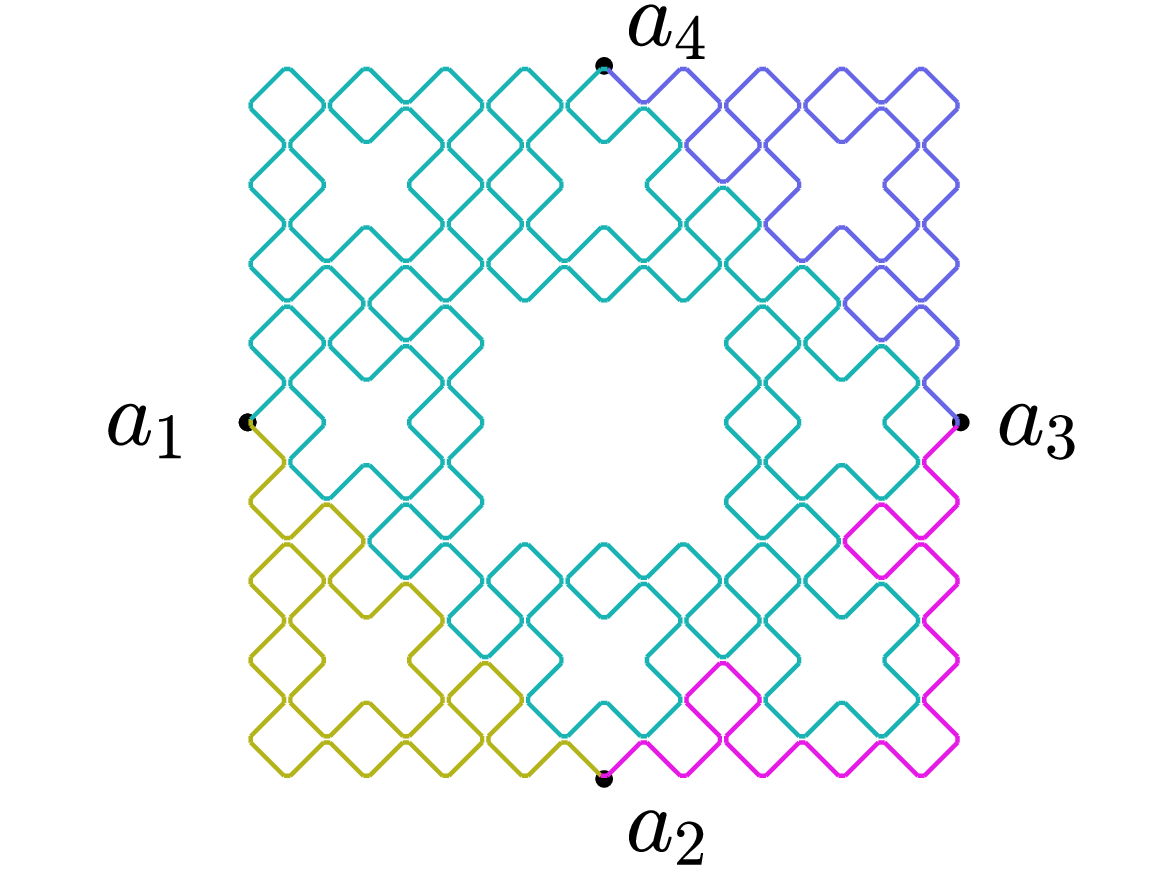}
\includegraphics[width=0.28\textwidth]{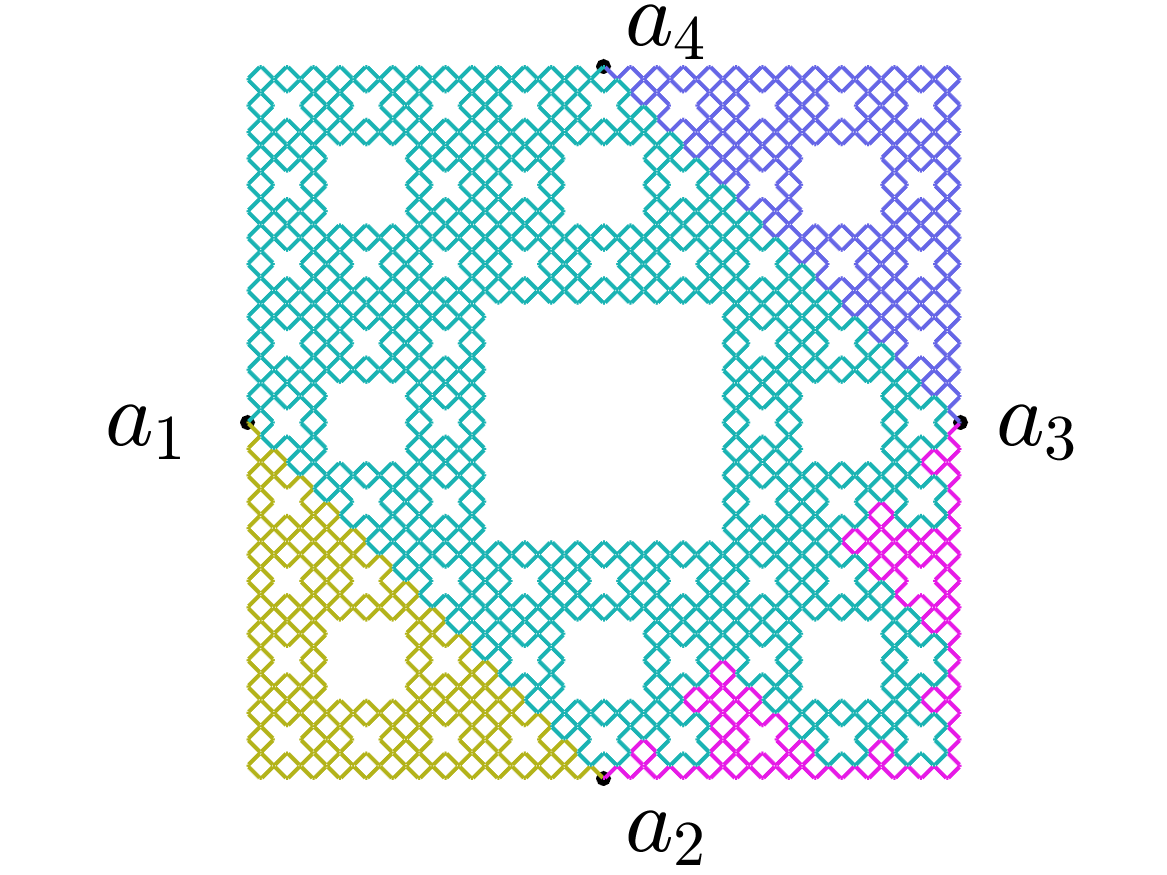}
\caption{The first three approximations of SFC of Sierpi\'nski carpet with middle points as a skeleton.}\label{CarpetSFC-3}
\end{figure}
\end{itemize}
\end{example}

\subsection*{Acknowledgement}
We thank the anonymous referees for valuable suggestions and comments, especially that concerns the self-similar zipper.
\bibliographystyle{siam}   %{vdC}
\bibliography{biblio}

\end{document}